%% file: Stable_processes_conditioned_to_be_continuously_absorbed_by_an_interval.tex
\documentclass[a4paper,12pt,oneside,reqno]{amsart}

\usepackage[german,english]{babel}

\usepackage{amsmath,amssymb,amsthm,enumerate,mathtools}

\usepackage{dsfont}
\usepackage{color}
\usepackage[utf8]{inputenc}
\usepackage[automark,nouppercase]{scrpage2}
\usepackage{hyperref}
\usepackage[numbers]{natbib}

\input{basis_paper.tex}

\allowdisplaybreaks

\title[S\MakeLowercase{table processes conditioned to hit an interval}]{Stable processes conditioned to hit an interval continuously from the outside}

\author{Leif D\"oring}
\address{Leif D\"oring: University of Mannheim, Institute of Mathematics, 68161 Mannheim, Germany.}
\email{doering@uni-mannheim.de}
\author{Philip Weißmann$^{*}$}\thanks{$^{*}$Supported by the Research Training Group "Statistical Modeling of Complex Systems" funded by the German Science Foundation}
\address{Philip Weißmann: University of Mannheim, Institute of Mathematics, 68161 Mannheim, Germany.}
\email{hweissma@mail.uni-mannheim.de}

\begin{document}
 
\begin{abstract}
	Conditioning stable L\'evy processes on zero probability events recently became a tractable subject since several explicit formulas emerged from a deep analysis using the Lamperti transformations for self-similar Markov processes. In this article we derive new harmonic functions and use them to explain how to condition stable processes to hit continuously a compact interval from the outside. 
\end{abstract}

\maketitle
\tableofcontents

\section{Introduction}\label{sec_intro}
We consider one-dimensional $\alpha$-stable L\'evy processes with scaling index $\alpha\in (0,2)$ killed on entering the interval $[-1,1]$. In Döring et al. \cite{Doer_Kyp_Wei_01} the authors found a positive invariant function (sometimes called positive harmonic function) for such killed processes, i.e. a function $h:\mR \setminus [-1,1] \rightarrow (0,\infty)$ such that
\begin{align}\label{invariant}
	\mE^x \big[ \1_{\{ t < T_{[-1,1]} \}} h(\xi_t) \big] = h(x), \quad x \notin [-1,1], t\geq 0,
\end{align}
where $T_{B} := \inf\left\lbrace t \geq 0: \xi_t \in  B \right\rbrace$ denotes the first hitting time of an open or closed set $B$. 
The invariant function was used to condition the stable processes to avoid the interval and to relate the conditioned processes to their $h$-transformed path measure:
$$\mE^x \Big[ \1_\Lambda \1_{\{ t < T_{[-1,1]} \}} \frac{h(\xi_t)}{h(x)} \Big] = \lim_{s \rightarrow \infty} \mP^x(\Lambda \, | \, s+t < T_{[-1,1]}), \quad x \notin [-1,1], t\geq 0,
$$
for $\Lambda \in \cF_t$, where $(\mathcal{F}_t)_{t \geq 0}$ is the natural enlargement of the filtration induced by $\xi$. As for other processes conditioned to avoid sets the conditioned stable processes are transient. As a counter part, the present article studies the question if stable processes can also be conditioned to hit the interval continuously in finite time. \smallskip

In recent years the problem of conditioning processes to hit a given set $B$ continuously has attracted some attention. As an example take a stable process of index smaller than $1$ and a singleton $B = \{0\}$. Moreover, denote $\rho:= \mP(\xi_1>0)$ and $\hat\rho := 1-\rho$. It was proved in Kyprianou et al. \cite{Kyp_Riv_Sat_01} that 
$$e:\mR \setminus \lc 0 \rc \rightarrow (0,\infty), x \mapsto \begin{cases}
\sin(\pi\alpha\hat\rho) x^{\alpha-1} \quad &\text{if } x>0 \\
\sin(\pi\alpha\rho) |x|^{\alpha-1} \quad &\text{if } x<0
\end{cases},
$$
is excessive for the killed process, i.e. 
$$\mE^x \big[ \1_{\{ t < T_{\{ 0 \}} \}} e(\xi_t) \big] \leq e(x), \quad x \neq 0, t \geq 0,$$
and the $h$-transform with $e$ coincides with the stable process conditioned to hit $0$ continuously. Indeed, the authors showed that the killing time is finite almost surely and the left-limit at the killing time is $0$. Applications of the conditioned processes have been found for instance in the study of entrance and exit at infinity of stochastic differential equations driven by stable processes, see Döring and Kyprianou \cite{Doer_Kyp_01}.\smallskip

In this article we will derive (strictly) excessive functions for the stable process killed on entering an interval, without loss of generality the interval $[-1,1]$, i.e. functions  $v: \mR \setminus [-1,1] \rightarrow (0,\infty)$ such that
\begin{align}\label{def_exc}
	\mE^x \big[ \1_{\{ t < T_{[-1,1]} \}} v(\xi_{t}) \big] \leq v(x), \quad x \notin [-1,1],t\geq 0.
\end{align}
Unfortunately, the corresponding $h$-transformed process is not self-similar and hence, we can not follow the strategy of \cite{Kyp_Riv_Sat_01} to show that this process hits the interval continuously. A second example for a process conditioned to be absorbed by a set due to Chaumont \cite{Chau_01} uses another way of showing continuous absorption and it will turn out that this way is the right one in our setting, too. Under some assumptions the author conditioned a Lévy process to be continuously absorbed by $0$ from above, i.e. to hit $(-\infty,0]$ continuously from the outside. The tool which was used is again a Doob $h$-transform with an excessive function $u:(0,\infty) \rightarrow (0,\infty)$ which has the additional condition that, for any compact $K \subseteq (0,\infty)$,
\begin{align}
\mE^x \big[ \1_{\{ T_{K^\mathrm{C}} < T_{(-\infty,0]} \}} u(\xi_{T_{K^\mathrm{C}}}) \big] = u(x), \quad x \notin [-1,1].
\end{align}
Such an excessive function is called harmonic. In Silverstein \cite{Sil_01} it was shown that in Chaumont's setting the role of $u$ is played by the potential density of the dual ladder height process. Considering the $h$-transformed process which we denote by $(\xi,\mP^x_{u})$ and the killing time by $\zeta$ one sees that
\begin{align}\label{h-trafo_Chau}
	\mP^x_{u}(T_{K^\mathrm{C}}<\zeta) = \mE^x \Big[ \1_{\{ T_{K^\mathrm{C}} < T_{(-\infty,0]} \}} \frac{u(\xi_{T_{K^\mathrm{C}}})}{u(x)} \Big] = 1.
\end{align}
This shows that the $h$-transformed process leaves all compact sets before it is killed. Chaumont even went further and extended (\ref{h-trafo_Chau}) to sets of the form $K=[a,\infty)$ which shows that the $h$-transformed process hits any set of the form $(0,a),a >0$, before killing, thus, absorption at $0$ is continuous.\smallskip

Before presenting our results, we introduce the most important definitions. More details can be found, for example, in Chung and Walsh \cite{Chu_Wal_01}, Bertoin \cite{Bert_01}, Kyprianou \cite{Kyp_03}
 or Sato \cite{Sat_01}.\smallskip

\textbf{Stable processes:} We consider the canonical process ${\xi}$ on the space of càdlàg paths equipped with the $\sigma$-algebra $\cF$ induced by the Skorohod topology. We denote by $\mP^x$ the probability measure on the path space that makes ${\xi}$ a stable process started from $x \in \mR$. Stable processes are L\'evy processes that fulfill the scaling property
\begin{align}\label{scaling}
\left((c {\xi}_{c^{-\alpha}t})_{t \geq 0},\mP^x \right) \overset{(d)}{=} \left(({\xi}_{t})_{t \geq 0},\mP^{cx}\right)
\end{align}
for all $x \in \mR$ and $c>0$, where $\alpha$ is the index of self-similarity. It turns out to be necessary that $\alpha\in (0,2]$ with $\alpha=2$ corresponding to the Brownian motion. The continuity of sample paths excludes the Brownian motion from our study so we restrict to $\alpha\in (0,2)$. As a Lévy process stable processes are characterised entirely by the L\'evy triplet. For $\alpha<2$, the linear and Brownian part vanish and the L\'evy measure is
\begin{align*}
	\Pi(\dd x) = \frac{\Gamma(\alpha+1)}{\pi} \left\{
\frac{ \sin(\pi \alpha \rho) }{ x^{\alpha+1}} \1_{\{x > 0\}} + \frac{\sin(\pi \alpha \hat\rho)}{ {|x|}^{\alpha+1} }\1_{\{x < 0\}}
 \right\}\dd x,\quad  x \in \mathbb{R},
\end{align*}
where 
$\rho:=\mP^0({\xi}_1 \geq 0)$ is the positivity parameter. For $\alpha \in (0,1)$ we exclude the case $\rho \in \lc 0,1 \rc$, in which case $\xi$ is (the negative of) a subordinator. For $\alpha \in (1,2)$ it is know that $\rho\in[{1}/{\alpha}, 1- {1}/{\alpha} ]$ and we exclude the boundary cases $\rho \in \left\lbrace {1}/{\alpha}, 1- {1}/{\alpha} \right\rbrace$ in which case $\xi$ has one-sided jumps.  For $\alpha=1$ we consider the symmetric Cauchy process excluding drift. The normalisation was chosen so that the characteristic exponent satisfies
\[
\mathbb{E}^x[{ e}^{{ i}\theta( {\xi}_1-x)}] = { e}^{-|\theta|^\alpha}, \quad \theta\in\mathbb{R}.
\]

An important fact we will use for the parameter regimes is that the stable process exhibits (set) transience and (set) recurrence according to whether $\alpha\in (0,1)$ or $\alpha\in[1,2)$. When $\alpha\in(1,2)$ the notion of recurrence is even stronger in the sense that fixed points are hit with probability one. \smallskip

\textbf{Killed L\'evy processes and $h$-transforms:}
The killed transition measures are defined as
$$p^{[-1,1]}_t(x,\dd y)=\mP^x(\xi_t \in \dd y, t< T_{[-1,1]}), \quad t\geq 0.$$
The corresponding sub-Markov process is called the L\'evy process killed in $[-1,1]$.
An excessive function for the killed process is a measurable function $v : \mR \backslash [-1,1] \rightarrow [0,\infty)$ such that
\begin{align}\label{eq_def_harm}
	\mE^x \big[ \1_{\{ t < T_{[-1,1]} \}} v(\xi_t) \big] \leq v(x),
	\quad x \in \mR \backslash [-1,1], t \geq 0.
\end{align}
An excessive function taking only strictly positive values is called a positive excessive function. When $v$ is a positive excessive function,  the associated Doob $h$-transform is defined via the change of measure
\begin{align}\label{def_htrafo}
	\mP_v^x(\Lambda, t < \zeta) := \mE^x \Big[ \1_\Lambda \1_{\{ t < T_{[-1,1]} \}} \frac{v(\xi_t)}{v(x)}  \Big],\quad x\in\mR \backslash [-1,1],
\end{align}
for $\Lambda \in \cF_t$, where $\zeta$ is the (possibly infinite) killing time of the process.
From Chapter 11 of Chung and Walsh \cite{Chu_Wal_01}, we know that under $\mP_v^{x}$ the canonical process is a strong Markov process and that (\ref{def_htrafo}) extends from deterministic times to $(\cF_t)_{t \geq 0}$-stopping times $T$;
that is,
\begin{align}\label{eq_htrafo_stoppingtime}
\mP_{v}^x(\Lambda, T <\zeta) = \mE^x \Big[\1_\Lambda \1_{\{ T<T_{[a,b]} \}} \frac{v(\xi_T)}{v(x)} \Big], \quad x \notin [-1,1],
\end{align}
for $\Lambda\in \mathcal F_T$. An excessive function $v: \mR \setminus [-1,1] \rightarrow (0,\infty)$ which fulfills
\begin{align}\label{def_harmonic}
\mE^x \Big[ \1_{\{ T_{K^\mathrm{C}} < T_{[-1,1]} \}} v(\xi_{T_{K^\mathrm{C}}}) \Big] = v(x), \quad x \notin [-1,1],
\end{align}
for all compact $K \subseteq \mR\setminus [-1,1]$ is called a positive harmonic function. 
\begin{remark}
The terminology of a harmonic function is not used consistently in the literature. In many articles (also including \cite{Doer_Kyp_Wei_01}) the notion of a harmonic function coincides with the notion of an invariant function in the sense of (\ref{invariant}). Here, we will always use the notion of a harmonic function for an excessive function which fulfills the additional condition \eqref{def_harmonic}. 
The crucial point on positive harmonic functions in the sense of this article is that the $h$-transform leaves all compact sets before being killed, see \eqref{h-trafo_Chau}.
\end{remark}

\section{Main results}
The main results of this article are two-fold. We first identify new harmonic functions (in the sense of \eqref{def_harmonic}) for the  stable processes killed in the unit interval. From these harmonic functions we define $h$-transformed measures which we then identify as the limiting measures of suitable conditionings that force the process to be absorbed at the boundary of the interval. The different possible cases of absorption at the top or the bottom of the interval will be reflected in the existence of different harmonic functions and their linear combinations.

\subsection{Harmonic functions} \label{sec_harm}
In this first section we identify two (minimal) harmonic functions. Let us define two functions $v_1, v_2: \mR \setminus [-1,1] \rightarrow (0,\infty)$ by 
$$v_1(x):= 
\begin{cases}
\sin(\pi\alpha\hat{\rho}) \Big[(x+1)\psi_{\alpha\rho}(x) - (\alpha-1)_+ \int\limits_1^x \psi_{\alpha\rho}(u) \, \dd u \Big]  \quad &\text{if } x>1 \\
\sin(\pi\alpha\rho) \Big[(|x|-1) \psi_{\alpha\hat{\rho}}(|x|) - (\alpha-1)_+ \int\limits_1^{|x|} \psi_{\alpha\hat{\rho}}(u) \, \dd u\Big]  \quad &\text{if } x<-1
\end{cases},
$$
and
$$v_{-1}(x):= 
\begin{cases}
\sin(\pi\alpha\hat{\rho}) \Big[(x-1) \psi_{\alpha\rho}(x) - (\alpha-1)_+ \int\limits_1^x \psi_{\alpha\rho}(u) \, \dd u \Big] \quad &\text{if } x>1 \\
\sin(\pi\alpha\rho) \Big[(|x|+1) \psi_{\alpha\hat{\rho}}(|x|) - (\alpha-1)_+ \int\limits_1^{|x|} \psi_{\alpha\hat{\rho}}(u) \, \dd u \Big]  \quad &\text{if } x<-1
\end{cases}.
$$
The appearing auxiliary functions
\begin{align*}
	\psi_{\alpha\rho}(x)= (x-1)^{\alpha\hat\rho-1}(x+1)^{\alpha\rho-1}, \quad x >1,
\end{align*}
already played a crucial rule to condition the stable processes to avoid an interval 	in \cite{Doer_Kyp_Wei_01}. For the function $\psi_{\alpha\hat\rho}$ the positivity parameter $\rho$ is replaced by $\hat\rho$, and vice versa.\smallskip

Here is the main result of this section:
\begin{theorem} \label{thm_harm}
Let $\xi$ be a stable process with index $\alpha \in (0,2)$ which has jumps in both directions. Then $v_1$ and $v_{-1}$ are harmonic functions for $\xi$ killed on first hitting the interval $[-1,1]$.
\end{theorem}
As described in the introduction a harmonic function is in particular excessive, hence, a new measure can be defined as an $h$-transform with the harmonic function. In what follows we will denote the $h$-transforms with $v_1$, $v_{-1}$ and $v:=v_1+v_{-1}$ by $\mP^x_{v_1}$, $\mP^x_{v_{-1}}$ and $\mP^x_{v}$.

\subsection{Stable processes absorbed from above (or below)} \label{sec_abs_above}
The purpose of this section is to analyse the $h$-transformed process $(\xi,\mP^x_{v_{1}})$. Since all results for $(\xi,\mP^x_{v_{-1}})$ are analogous (replacing $\rho$ and $\hat \rho$) without loss of generality we only discuss $(\xi,\mP^x_{v_{1}})$. Two questions will be our main concern:
\begin{itemize}
	\item Is the process killed in finite time and, if so, what is the limiting behavior at the killing time?
	\item How to characterize $\mP^x_{v_{1}}$ through a limiting conditioning of $\mP^x$?
\end{itemize}
The first question can be answered for all $\alpha$ simultaneously using properties of the explicit form of $v_1$:
\begin{proposition}\label{thm_abs_above}
Let $\xi$ be an $\alpha$-stable process with $\alpha \in (0,2)$ and both sided jumps, then
$$\mP_{v_1}^x(\zeta < \infty, \xi_{\zeta-} =1)=1, \quad x \notin [-1,1].$$
\end{proposition}
To answer the second question we need to distinguish the recurrent and the transient cases:\\
\textbf{The case $\alpha<1$:} The probability that $\xi$ never hits the interval $[-1,1]$ is positive because the stable process is transient. To condition $\xi$ to be absorbed by $[-1,1]$ from above without hitting the interval we first condition on $\{T_{[-1,1]}=\infty\}$ and then on some event which describes the absorption from above. The most plausible event is $T_{(1,1+\eps)}$ being finite for small $\eps>0$. Another possibility refers to the so-called point of closest reach. Let therefore $\underline{m}$ be the time such that $|\xi_{\underline{m}}| \leq |\xi_t|$ for all $t \geq 0$. Then $\xi_{\underline{m}}$ is called the point of closest reach of $0$. The polarity of points for $\alpha<1$ implies $\xi_{\underline{m}} \neq 0$ almost surely under $\mP^x$ for all starting points $x \neq 0$. With these definitions one could also think of conditioning on the event $\{ \xi_{\underline{m}} \in (1,1+\eps) \}$ which is contained in $\{ T_{[-1,1]}=\infty, T_{(1,1+\eps)}<\infty \}$ and, indeed, this is the right choice. \smallskip

\textbf{The case $\alpha \geq 1$:} The first hitting time $T_{[-1,1]}$ is finite almost surely, hence, a different conditioning is needed. Since $T_{(-1-\eps,1+\eps)}$ is finite as well the good conditioning is to condition $\xi_{T_{(-1-\eps,1+\eps)}}$ to be in $(1,1+\eps)$ and then let $\eps$ tend to $0$.\smallskip

The techniques we use for the conditioning center around the recent results on the so-called deep factorisation of stable processes, see e.g. Kyprianou \cite{Kyp_02} and Kyprianou et al. \cite{Kyp_Riv_Sen_01} and hitting distributions of stable processes, see Kyprianou et al. \cite{Kyp_Par_Wat_01}. In particular, results on the distribution of the point of closest reach in the case $\alpha<1$ and the distribution of the first hitting time of the interval $(-1,1)$ in the case $\alpha \geq 1$ are the keys to prove our results.\smallskip

We come to the first characterisation of the $h$-transform $\mP_{v_1}^x$ as the process conditioned to be absorbed by $[-1,1]$ from above in a meaningful way.
\begin{theorem}\label{thm_cond_v1_<1}
Let $\xi$ be an $\alpha$-stable process with $\alpha\in (0,1)$ and both sided jumps. Then it holds, for all $x \notin [-1,1]$ and $\Lambda \in \cF_t$, that
$$\mP^{x}_{v_1}(\Lambda, t < \zeta) =  \lim_{\delta \searrow 0}\lim_{\eps \searrow 0} \mP^x(\Lambda, t< T_{(-(1+\delta),1+\delta)} \,|\, \xi_{\underline{m}} \in (1,1+\eps)).$$
\end{theorem}
In fact, we prove a slightly more general statement which has precisely the form of a self-similar Markov process conditioned to be absorbed at the origin in Kyprianou et al. \cite{Kyp_Riv_Sat_01} and a L\'evy process conditioned to be absorbed at the origin from above in Chaumont \cite{Chau_01}:
\begin{align}\label{aa}
	\mP^{x}_{v_1}(\Lambda, t < T_{(-(1+\delta),1+\delta)}) = \lim_{\eps \searrow 0} \mP^x(\Lambda, t< T_{(-(1+\delta),1+\delta)} \,|\, \xi_{\underline{m}} \in (1,1+\eps))
\end{align}
for all $\delta>0$.\smallskip


In the case $\alpha \geq 1$ the $h$-transform belongs to a different conditioned process.
\begin{theorem}\label{thm_cond_v1_geq1}
Let $\xi$ be an $\alpha$-stable process with $\alpha \in [1,2)$ and both sided jumps. Then it holds, for all $x \notin [-1,1]$ and $\Lambda \in \cF_t$, that
$$\mP^{x}_{v_1}(\Lambda, t < \zeta) = \lim_{\delta \searrow 0} \lim_{\eps \searrow 0} \mP^x(\Lambda, t< T_{(-(1+\delta),1+\delta)} \,|\xi_{T_{(-(1+\eps),1+\eps)}} \in (1,1+\eps)).$$
\end{theorem}

With this result we can interpret the $h$-transformed process as the original process conditioned to approach the interval $[-1,1]$ continuously from above.\smallskip

For $\alpha >1$ we can even find a second characterisation of $\mP^x_{v_1}$ as conditioned process. We need to introduce the stable process conditioned to avoid $0$ (see e.g. Pantí \cite{Pan_01} or Yano \cite{Yan_01} for general Lévy processes). For this sake define $e: \mR \setminus \lc 0 \rc \rightarrow (0,\infty)$ via
$$ 
e(x) = \begin{cases}
\sin(\pi\alpha\hat\rho) x^{\alpha-1} \quad &\text{if } x>0 \\
\sin(\pi\alpha\rho) |x|^{\alpha-1} \quad &\text{if } x<0
\end{cases},
$$
which is known to be a positive invariant function for the process killed on hitting $0$ when $\alpha >1$. Denote the underlying $h$-transform by $\mP^x_\circ$, i.e. 
$$\mP^x_\circ(\Lambda) = \mE^x \Big[\1_{\{ t < T_{\{0\}} \}} \frac{e(\xi_t)}{e(x)} \Big], \quad x \neq 0, \Lambda \in \cF_t,$$
which can be shown to correspond to conditioning the stable process to avoid the origin. We can use $\mP^x_\circ$ to give a conditioning analogously to the case $\alpha<1$ also in the case $\alpha >1$. But here the conditioning does not refer to the original process but to the process conditioned to avoid $0$. 

\begin{theorem}\label{thm_cond_v1_altern_>1}
Let $\xi$ be an $\alpha$-stable process with $\alpha \in (1,2)$ and both sided jumps. Then it holds, for all $x \notin [-1,1]$ and $\Lambda \in \cF_t$, that
$$\mP^{x}_{v_1}(\Lambda, t < \zeta) = \lim_{\delta \searrow 0} \lim_{\eps \searrow 0} \mP_\circ^x(\Lambda, t< T_{(-(1+\delta),1+\delta)} \,|\,\xi_{\underline{m}} \in (1,1+\eps)).$$
\end{theorem}
It is quite remarkable to compare Theorem \ref{thm_cond_v1_altern_>1} and Theorem \ref{thm_cond_v1_<1}. Since conditioning to avoid a point has no effect for $\alpha< 1$ both theorems coincide. First condition to avoid the origin (trivial for $\alpha< 1$) then condition to approach $1$ from above yields $\mP^{x}_{v_1}$. The case $\alpha=1$ differs from $\alpha \neq 1$ in this respect because $0$ is polar and the conditioning to approach the interval from above is not well-defined because $\xi_{\underline m}=0$ almost surely.

\subsection{Stable processes absorbed without restrictions} \label{sec_abs_both}
In this section we want to analyse the h-transforms $(\xi,\mP^x_{v})$ with $v=v_1+v_{-1}$. The two main aspects are the same as in Section \ref{sec_abs_above}. First we want to analyse the behaviour of the paths of $(\xi,\mP^x_v)$ at the killing time if it is finite. Second we give characterisations of the $h$-transformed process as the original process conditioned on similar events as in Section \ref{sec_abs_above}.\smallskip

In the case $\alpha<1$ this works as one would expect, namely the $h$-transform using $v$ corresponds to the process conditioned on $\{ |\xi_{\underline{m}}| \in (1,1+\eps) \}$ for $\eps$ tending to $0$. For $\alpha \geq 1$ we won't find a representation $(\xi,\mP^x_{v})$ as a conditioned process. Nonetheless we can show that the process conditioned to be absorbed by $[-1,1]$ without any restrictions on the side of the interval of which it is absorbed, equals $(\xi,\mP^x_{v_1})$ or  $(\xi,\mP^x_{v_{-1}})$ depending on some relation on $\rho$. This means that the process conditioned to be absorbed without any restrictions coincides with one of the processes conditioned to be absorbed from one side.\smallskip


Here is the result on the behaviour at the killing time:
\begin{proposition}\label{thm_abs_both}
Let $\xi$ be an $\alpha$-stable process with $\alpha \in (0,2)$ and both sided jumps, then
$$\mP_{v}^x(\zeta < \infty, |\xi_{\zeta-}| =1)=1, \quad x \notin [-1,1].$$
\end{proposition}

As before we want to connect the $h$-transformed process to some conditioned process. Again we have to separate the cases $\alpha<1$ and $\alpha \geq 1$ and for $\alpha>1$ we give an alternative conditioned process. The event we condition on is bigger than in Section \ref{sec_abs_above} in all cases.

We start with the asymptotic in the case $\alpha<1$ and the characterisation of $(\xi,\mP^x_v)$ as conditioned process as one would expect with the knowledge of Theorem \ref{thm_cond_v1_<1}.
\begin{theorem}\label{thm_cond_v_<1}
Let $\xi$ be an $\alpha$-stable process with $\alpha \in (0,1)$ and both sided jumps. Then it holds, for all $x \notin [-1,1]$ and $\Lambda \in \cF_t$, that
$$\mP^{x}_{v}(\Lambda, t < \zeta) =  \lim_{\delta \searrow 0} \lim_{\eps \searrow 0} \mP^x(\Lambda, t< T_{(-(1+\delta),1+\delta)} \,|\, |\xi_{\underline{m}}| \in (1,1+\eps)).$$
\end{theorem}

As we already mentioned, in the case $\alpha \geq 1$ the process conditioned to be absorbed by the interval without restriction on the side of absorption is the same as the process conditioned to be absorbed from one side, the side depending on $\rho$.
\begin{theorem}\label{thm_cond_v_geq1}
Let $\xi$ be an $\alpha$-stable process with $\alpha \in [1,2)$ and both sided jumps. Then it holds, for all $x \notin [-1,1]$ and $\Lambda \in \cF_t$, that
\begin{align*}
& \lim_{\delta \searrow 0} \lim_{\eps \searrow 0} \mP^x(\Lambda, t< T_{(-(1+\delta),1+\delta)} \,| \, \xi_{T_{(-(1+\eps),1+\eps)}} \notin [-1,1])\\
=&\, \begin{cases}
\mP^{x}_{v_1}(\Lambda, t < \zeta) \quad &\text{if } \rho \leq \frac{1}{2} \\
\mP^{x}_{v_{-1}}(\Lambda, t < \zeta) \quad &\text{if } \rho > \frac{1}{2}
\end{cases}.
\end{align*}
\end{theorem}

We conclude with the alternative characterisation for the $h$-transform for $\alpha>1$. Again the conditioning refers to the stable process conditioned to avoid $0$ and the event we condition on is the same as in the case $\alpha<1$.
\begin{theorem}\label{thm_cond_v_altern_>1}
Let $\xi$ be an $\alpha$-stable process with $\alpha \in (1,2)$ and both sided jumps. Then it holds, for all $x \notin [-1,1]$ and $\Lambda \in \cF_t$, that
$$\mP^{x}_{v}(\Lambda, t < \zeta) =  \lim_{\delta \searrow 0}\lim_{\eps \searrow 0} \mP_\circ^x(\Lambda, t< T_{(-(1+\delta),1+\delta)} \,|\,|\xi_{\underline{m}}| \in (1,1+\eps)).$$
\end{theorem}

\section{Proofs}\label{sec_proofs}

\subsection{Harmonic functions} \label{sec_proof_harmonic}
In this section we prove Theorem \ref{thm_harm}. First we give an idea how to extract the right harmonic functions. The potential measure of $\xi$ killed when it enters $[-1,1]$ is defined as
$$U_{[-1,1]}(x,\dd y) := \mE^x \Bigg[ \int\limits_0^{T_{[-1,1]}} \1_{\{ \xi_t \in \dd y \}} \, \dd t \Bigg], \quad x,y \notin [-1,1].$$
It is known that the potential measure has a density with respect to the Lebesgue measure (also known as Green's function), i.e.
$$U_{[-1,1]}(x,\dd y) = u_{[-1,1]}(x,y) \, \dd y,$$
where $u_{[-1,1]}: (\mR\setminus [-1,1])^2 \rightarrow [0,\infty)$ is explicitely known from Profeta and Simon \cite{Pro_Sim_01}. Moreover, Kunita and Watanabe \cite{Kun_Wat_01} showed that $x \mapsto u_{[-1,1]}(x,y)$ is harmonic for all $y \notin [-1,1]$ and, heuristically speaking, the corresponding $h$-transform should be the process conditioned to be absorbed by $y$. Since our aim is to condition the process to be absorbed from $1$ we will consider the limit when $y$ tends to $1$. But from the formulas of \cite{Pro_Sim_01} we see immediately that $u_{[-1,1]}(x,y)$ converges to $0$ for $y$ tending to $1$. So there are two difficulties. The first one is that we need to renormalise $u_{[-1,1]}(x,y)$ such that it converges pointwise for $y \searrow 1$ to some function in $x$ and second we need to argue why in this case the limit of the (scaled) harmonic function is harmonic again.\smallskip

To abbreviate we denote
$$c_{\alpha\rho} :=  2^{\alpha\rho}\frac{\pi \alpha\rho\Gamma(\alpha\rho)}{\Gamma(1-\alpha\hat{\rho})} \quad \text{and} \quad c_{\alpha\hat\rho}:= 2^{\alpha\hat{\rho}}\frac{\pi \alpha\hat{\rho}\Gamma(\alpha\hat{\rho})}{\Gamma(1-\alpha\rho)}.$$
The first auxiliary result establishes a pointwise connection between $v_1$ and the potential density $u_{[-1,1]}$ which will be very important for the proof of harmonicity of $v_1$. From Profeta and Simon \cite{Pro_Sim_01} we know that $y \mapsto u_{[-1,1]}(x,y)$ has a pole in $x$ (for $\alpha<1$) but is also integrable at $x$. Hence, defining $u_{[-1,1]}(x,x) := 0$ does not change anything for the potential of the process killed on entering $[-1,1]$.
\begin{lemma} \label{lemma_boring}
Whenever $x>y>1$ or $x<-1, y>1$, it holds that
\begin{align*}
v_1(x) &= 2^{\alpha\hat\rho-1} c_{\alpha\rho} \frac{u_{[-1,1]}(x,y)}{g(y)}\\
&\quad -\Big(\sin(\pi\alpha\hat\rho) \1_{\lc x>1 \rc} + \sin(\pi\alpha\rho) \1_{\lc x<-1 \rc}\Big)\\
& \quad  \quad \times \frac{(1-\alpha\hat\rho)|x-y|^{\alpha-1}}{g(y)} \int\limits_1^{z(x,y)} (u-1)^{\alpha\rho}(u+1)^{\alpha\hat\rho-2}\, \dd u
\\
&\quad +(\alpha-1)_+ \Big(\sin(\pi\alpha\hat\rho)\1_{\{ x>1 \}}\int\limits_1^x \psi_{\alpha\rho}(u)\, \dd u  + \sin(\pi\alpha\rho) \1_{\{ x<-1 \}}\int\limits_1^{|x|} \psi_{\alpha\hat\rho}(u)\, \dd u\Big)\\
&\quad  \quad \times \Big(\frac{\alpha\rho}{g(y)} \int\limits_1^y \psi_{\alpha\hat\rho}(u)\,\dd u -1 \Big),
\end{align*}
where $g(y)=(y-1)^{\alpha\rho}(y+1)^{\alpha\hat\rho-1} = (y-1) \psi_{\alpha\hat\rho}(y)$.
\end{lemma}

\begin{proof}
We use the explicit expression for $u_{[-1,1]}(x,y)$ from Profeta and Simon \cite{Pro_Sim_01}, where the expression
$$z(x,y) _= \frac{|xy-1|}{|x-y|}, \quad x,y \notin [-1,1],x \neq y,$$
appears frequently. Before we start we note that
$$z(x,y)-1 = \begin{cases}
\frac{(x+1)(y-1)}{x-y} \quad &\text{if } x>y>1 \\
\frac{(|x|-1)(y-1)}{y-x} \quad &\text{if } x<-1,y>1
\end{cases}$$
and 
$$z(x,y)+1 = \begin{cases}
\frac{(x-1)(y+1)}{x-y} \quad &\text{if } x>y>1 \\
\frac{(|x|+1)(y+1)}{y-x} \quad &\text{if } x<-1,y>1
\end{cases}.$$
Furthermore, with integration by parts we get
\begin{align*}
&\quad\int\limits_1^{z(x,y)} \psi_{\alpha\rho}(u) \, \dd u\\
&= \int\limits_1^{z(x,y)} (u-1)^{\alpha\hat{\rho}-1} (u+1)^{\alpha\rho-1} \, \dd u \\
&= \frac{1}{\alpha\hat{\rho}}\Big[(u-1)^{\alpha\hat{\rho}} (u+1)^{\alpha\rho-1}\Big]_1^{z(x,y)} - \frac{\alpha\rho-1}{\alpha\hat{\rho}}\int\limits_1^{z(x,y)} (u-1)^{\alpha\hat{\rho}} (u+1)^{\alpha\rho-2} \, \dd u \\
&= \frac{1}{\alpha\hat{\rho}}\Big((z(x,y)-1)^{\alpha\hat{\rho}} (z(x,y)+1)^{\alpha\rho-1}\Big) + \frac{1-\alpha\rho}{\alpha\hat{\rho}}\int\limits_1^{z(x,y)} (u-1)^{\alpha\hat{\rho}} (u+1)^{\alpha\rho-2} \, \dd u
\end{align*}
and analogously
\begin{align*}
&\quad\int\limits_1^{z(x,y)} \psi_{\alpha\hat\rho}(u) \, \dd u\\
&= \frac{1}{\alpha\rho}\Big((z(x,y)-1)^{\alpha\rho} (z(x,y)+1)^{\alpha\hat\rho-1}\Big) + \frac{1-\alpha\hat\rho}{\alpha\rho}\int\limits_1^{z(x,y)} (u-1)^{\alpha\rho} (u+1)^{\alpha\hat\rho-2} \, \dd u.
\end{align*}
We use the explicit form for $u_{[-1,1]}(x,y)$ given in \cite{Pro_Sim_01} and plug in to see, for $x>y>1$,
\begin{align*}
&\quad\frac{\Gamma(\alpha\rho)\Gamma(\alpha\hat{\rho})}{2^{1-\alpha}} u_{[-1,1]}(x,y)\\
&= (x-y)^{\alpha-1} \int\limits_1^{z(x,y)} \psi_{\alpha\hat{\rho}}(u) \, \dd u - (\alpha-1)_+  \int\limits_1^y \psi_{\alpha\hat\rho}(u) \, \dd u \int\limits_1^x \psi_{\alpha\rho}(u) \, \dd u \\
&= \frac{(x-y)^{\alpha-1}}{\alpha\rho}(z(x,y)-1)^{\alpha\rho} (z(x,y)+1)^{\alpha\hat\rho-1}\\
&\quad + \frac{(1-\alpha\hat\rho)(x-y)^{\alpha-1}}{\alpha\rho}\int\limits_1^{z(x,y)} (u-1)^{\alpha\rho} (u+1)^{\alpha\hat\rho-2} \, \dd u \\
&\quad - (\alpha-1)_+  \int\limits_1^y \psi_{\alpha\hat\rho}(u) \, \dd u \int\limits_1^x \psi_{\alpha\rho}(u) \, \dd u \\
&= \frac{1}{\alpha\rho}((x+1)(y-1))^{\alpha\rho} ((x-1)(y+1))^{\alpha\hat\rho-1}\\
&\quad + \frac{(1-\alpha\hat\rho)(x-y)^{\alpha-1}}{\alpha\rho}\int\limits_1^{z(x,y)} (u-1)^{\alpha\rho} (u+1)^{\alpha\hat\rho-2} \, \dd u \\
&\quad - (\alpha-1)_+  \int\limits_1^y \psi_{\alpha\hat\rho}(u) \, \dd u \int\limits_1^x \psi_{\alpha\rho}(u) \, \dd u \\
&= \frac{1}{\alpha\rho}(y-1)^{\alpha\rho} (y+1)^{\alpha\hat\rho-1} \Big(\frac{1}{\sin(\pi\alpha\hat\rho)} v_1(x)+(\alpha-1)_+ \int\limits_1^x \psi_{\alpha\rho}(u)\, \dd u \Big)\\
&\quad + \frac{(1-\alpha\hat\rho)(x-y)^{\alpha-1}}{\alpha\rho}\int\limits_1^{z(x,y)} (u-1)^{\alpha\rho} (u+1)^{\alpha\hat\rho-2} \, \dd u \\
&\quad - (\alpha-1)_+  \int\limits_1^y \psi_{\alpha\hat\rho}(u) \, \dd u \int\limits_1^x \psi_{\alpha\rho}(u) \, \dd u.
\end{align*}
Solving the equation with respect to $v_1$ and using $\sin(\pi\alpha\hat\rho)= \frac{\pi}{\Gamma(\alpha\hat\rho)\Gamma(1-\alpha\hat\rho)}$ yields the claim for $x>y>1$. For $x<-1,y>1$ we get similarly:
\begin{align*}
&\quad\frac{\sin(\pi\alpha\hat\rho)}{\sin(\pi\alpha\rho)}\frac{\Gamma(\alpha\rho)\Gamma(\alpha\hat{\rho})}{2^{1-\alpha}} u_{[-1,1]}(x,y)\\
&= (y-x)^{\alpha-1} \int\limits_1^{z(x,y)} \psi_{\alpha\hat{\rho}}(u) \, \dd u - (\alpha-1)_+  \int\limits_1^y \psi_{\alpha\hat\rho}(u) \, \dd u \int\limits_1^{|x|} \psi_{\alpha\hat\rho}(u) \, \dd u \\
&= \frac{(y-x)^{\alpha-1}}{\alpha\rho}(z(x,y)-1)^{\alpha\rho} (z(x,y)+1)^{\alpha\hat\rho-1}\\
&\quad + \frac{(1-\alpha\hat\rho)(y-x)^{\alpha-1}}{\alpha\rho}\int\limits_1^{z(x,y)} (u-1)^{\alpha\rho} (u+1)^{\alpha\hat\rho-2} \, \dd u \\
&\quad - (\alpha-1)_+  \int\limits_1^y \psi_{\alpha\hat\rho}(u) \, \dd u \int\limits_1^{|x|} \psi_{\alpha\hat\rho}(u) \, \dd u \\
&= \frac{1}{\alpha\rho}((|x|-1)(y-1))^{\alpha\rho} ((|x|+1)(y+1))^{\alpha\hat\rho-1}\\
&\quad + \frac{(1-\alpha\hat\rho)(y-x)^{\alpha-1}}{\alpha\rho}\int\limits_1^{z(x,y)} (u-1)^{\alpha\rho} (u+1)^{\alpha\hat\rho-2} \, \dd u \\
&\quad - (\alpha-1)_+  \int\limits_1^y \psi_{\alpha\hat\rho}(u) \, \dd u \int\limits_1^{|x|} \psi_{\alpha\hat\rho}(u) \, \dd u \\
&= \frac{1}{\alpha\rho}(y-1)^{\alpha\rho} (y+1)^{\alpha\hat\rho-1} \Big(\frac{1}{\sin(\pi\alpha\rho)} v_1(x)+(\alpha-1)_+ \int\limits_1^{|x|} \psi_{\alpha\hat\rho}(u)\, \dd u \Big)\\
&\quad + \frac{(1-\alpha\hat\rho)(x-y)^{\alpha-1}}{\alpha\rho}\int\limits_1^{z(x,y)} (u-1)^{\alpha\rho} (u+1)^{\alpha\hat\rho-2} \, \dd u \\
&\quad - (\alpha-1)_+  \int\limits_1^y \psi_{\alpha\hat\rho}(u) \, \dd u \int\limits_1^{|x|} \psi_{\alpha\hat\rho}(u) \, \dd u.
\end{align*}
Again, solving with respect to $v_1(x)$ leads to the claim.
\end{proof}

\begin{corollary}\label{cor_limit}
It holds that
$$v_1(x) =  c_{\alpha\rho}\lim_{y \searrow 1} \frac{u_{[-1,1]}(x,y)}{(y-1)^{\alpha\rho}},\quad x \in \mR \setminus [-1,1].$$
\end{corollary}
\begin{proof}
We consider the expression from Lemma \ref{lemma_boring} and let $y$ tend to $1$ from above. It is sufficient to show that
\begin{align*}
&-\Big(\sin(\pi\alpha\hat\rho) \1_{\lc x>1 \rc} + \sin(\pi\alpha\rho) \1_{\lc x<-1 \rc}\Big)\\
&\quad   \times \frac{(1-\alpha\hat\rho)|x-y|^{\alpha-1}}{g(y)} \int\limits_1^{z(x,y)} (u-1)^{\alpha\rho}(u+1)^{\alpha\hat\rho-2}\, \dd u
\\
&+(\alpha-1)_+ \Big(\sin(\pi\alpha\hat\rho)\1_{\lc x>1 \rc}\int\limits_1^x \psi_{\alpha\rho}(u)\, \dd u  + \sin(\pi\alpha\rho) \1_{\lc x<-1 \rc}\int\limits_1^{|x|} \psi_{\alpha\hat\rho}(u)\, \dd u\Big)\\
&\quad   \times \Big(\frac{\alpha\rho}{g(y)} \int\limits_1^y \psi_{\alpha\hat\rho}(u)\,\dd u -1 \Big)
\end{align*}
converges to $0$ for $y \searrow 1$. For that it is of course sufficient to show that 
$$\frac{1}{g(y)} \int\limits_1^{z(x,y)} (u-1)^{\alpha\rho}(u+1)^{\alpha\hat\rho-2}\, \dd u \quad
\text{ and } \quad \frac{\alpha\rho}{g(y)} \int\limits_1^y \psi_{\alpha\hat\rho}(u)\,\dd u -1 $$
converge to $0$ for $y \searrow 1$. Both claims can be seen readily with l'Hopital's rule.
\end{proof}

Now we prove harmonicity of $v_1$.
\begin{proof}[Proof of Theorem \ref{thm_harm}]
To show excessiveness we define the measure
$$\eta(\dd x) \coloneqq v_1(x) \, \dd x\quad \text{ on } \mR \setminus [-1,1].$$
We will show that $\eta$ is an excessive measure for the dual process killed on entering the intervall, i.e. $\eta$ is $\sigma$-finite and it holds that
$$\int\limits_{\mR\setminus [-1,1]} \hat{\mP}^x(\xi_t \in A, t < T_{[-1,1]}) \, \eta(\dd x) \leq \eta(A),$$
for all $A \in \cB(\mR \setminus [-1,1])$ and $t \geq 0$. From Theorem XII.71 of Dellacherie and Meyer \cite{Del_Mey_01} it is known that if an excessive measure has a density with respect to the duality measure (which is the Lebesgue measure also for killed Lévy processes, see Bertoin \cite{Bert_01}, Theorem II.5), then this density is an excessive function for the dual process killed on hitting $[-1,1]$. Hence, by showing that $\eta$ is an excessive measure for the dual process killed on hitting $[-1,1]$, it follows that $v_1$ is an excessive function for the original process killed on entering the interval.\smallskip

To show that $\eta$ is excessive for the dual process, first note that $\eta$ is $\sigma$-finite because $v_1$ is continuous on $\mR \setminus [-1,1]$. Next, for the dual process, we note that
$$\hat{U}_{[-1,1]}(y, \dd x) = u_{[-1,1]}(x,y) \, \dd x,\quad x,y \in \mR \setminus [-1,1],$$ 
where $\hat{U}_{[-1,1]}$ is the potential of the dual process killed on entering $[-1,1]$ (see Theorem XII.72 of Dellacherie and Meyer \cite{Del_Mey_01} for a general Markov process). Let $A \in \cB(\mR \setminus [-1,1])$ be compact, use Corollary \ref{cor_limit} in the first equation and Fatou's Lemma in the second one:
\begin{align*}
&\quad \frac{1}{c_{\alpha\rho}} \int\limits_{\mR\setminus [-1,1]} \hat{\mP}^x(\xi_t \in A, t < T_{[-1,1]}) \, \eta(\dd x)\\
&= \int\limits_{\mR\setminus [-1,1]} \hat{\mP}^x(\xi_t \in A, t < T_{[-1,1]})  \lim\limits_{y \searrow 1} \frac{u_{[-1,1]}(x,y)}{(y-1)^{\alpha\rho}} \,\dd x\\
&\leq \liminf\limits_{y \searrow 1} \frac{1}{(y-1)^{\alpha\rho}} \int\limits_{\mR\setminus [-1,1]} \hat{\mP}^x(\xi_t \in A, t < T_{[-1,1]}) u_{[-1,1]}(x,y) \, \dd x \\
&\leq \liminf\limits_{y \searrow 1} \frac{1}{(y-1)^{\alpha\rho}} \int\limits_{\mR\setminus [-1,1]} \hat{\mP}^x(\xi_t \in A, t < T_{[-1,1]}) \, \hat{U}_{[-1,1]}(y,\dd x)\\
&= \liminf_{y \searrow 1}\frac{1}{(y-1)^{\alpha\rho}}\int\limits_0^\infty \Big( \int\limits_{\mR\setminus [-1,1]} \hat{\mP}^x(\xi_t \in A, t < T_{[-1,1]}) \,  \hat{\mP}^y(\xi_s \in \dd x, s < T_{[-1,1]}) \Big) \, \dd s\\
&= \liminf_{y \searrow 1}\frac{1}{(y-1)^{\alpha\rho}}\int\limits_0^\infty \hat{\mP}^y(\xi_{t+s} \in A, t+s < T_{[-1,1]}) \, \dd s \\
&= \liminf_{y \searrow 1}\frac{1}{(y-1)^{\alpha\rho}}\int\limits_t^\infty \hat{\mP}^y(\xi_{s} \in A, s < T_{[-1,1]}) \, \dd s \\
&\leq \liminf_{y \searrow 1}\frac{1}{(y-1)^{\alpha\rho}}\int\limits_0^\infty \hat{\mP}^y(\xi_{s} \in A, s < T_{[-1,1]}) \, \dd s \\
&\leq \liminf_{y \searrow 1}\frac{1}{(y-1)^{\alpha\rho}}\int\limits_A \hat{u}_{[-1,1]}(y,x) \, \dd x \\
&\leq \liminf_{y \searrow 1}\int\limits_A \frac{u_{[-1,1]}(x,y)}{(y-1)^{\alpha\rho}}  \, \dd x.
\end{align*}
From Corollary \ref{cor_limit} we know that $(u_{[-1,1]}(x,y))/((y-1)^{\alpha\rho})$ converges for $y \searrow 1$ for all $x \in \mR \setminus [-1,1]$, in particular the function $y \mapsto (u_{[-1,1]}(x,y))/((y-1)^{\alpha\rho})$ is bounded on $(1,\eps)$ with $\eps < \inf A \cap (1,\infty)$ for all $x \in A$. But since $A$ is compact $(u_{[-1,1]}(x,y))/((y-1)^{\alpha\rho})$ is uniformly bounded for $x \in A$. Hence, we can apply dominated convergence to deduce:
\begin{align*}
\frac{1}{c_{\alpha\rho}} \int\limits_{\mR\setminus [-1,1]} \hat{\mP}^x(\xi_t \in A) \, \eta(\dd x) &\leq \int\limits_A \lim_{y \searrow 1} \frac{u_{[-1,1]}(x,y)}{(y-1)^{\alpha\rho}} \, \dd x\\
&= \frac{1}{c_{\alpha\rho}} \int_A v_1(x) \, \dd x\\
&= \frac{1}{c_{\alpha\rho}} \eta(A).
\end{align*}
Hence, we proved that $\eta$ is an excessive measure and as mentioned above it follows with Theorem XII.71 of \cite{Del_Mey_01} that $v_1$ is an excessive function.\smallskip

Now we show the characterising property of harmonicty, i.e.
$$\mE^x\big[ \1_{\{ T_{K^\mathrm{C}} < T_{[-1,1]} \}} v_1(\xi_{T_{K^\mathrm{C}}}) \big] = v_1(x), \quad x \in \mR\setminus [-1,1],$$
for all $K \subseteq \mR\setminus [-1,1]$ which are compact in $\mR\setminus [-1,1]$. If $x \in K^\mathrm{C}=(\mR\setminus[-1,1]) \setminus K$ the claim is clear. So we assume $x \in K$. The idea is to use the connection between $v_1$ and $u_{[-1,1]}$ from Lemma \ref{lemma_boring} and Proposition 6.2 (ii) of Kunita and Watanabe \cite{Kun_Wat_01}. The second tells us that $x \mapsto u_{[-1,1]}(x,y)$ is harmonic on $(\mR\setminus [-1,1])\setminus \lc y \rc$ for all $y \in \mR\setminus [-1,1]$, i.e.
$$\mE^x\big[ \1_{\{ T_{K^\mathrm{C}} < T_{[-1,1]} \}} u_{[-1,1]}(\xi_{T_{K^\mathrm{C}}},y) \big] = u_{[-1,1]}(x,y), \quad x,y \in \mR\setminus [-1,1], x \neq y,$$
for all $K \subseteq \mR\setminus [-1,1]$ which are compact in $ \mR\setminus [-1,1] \setminus \{ y \}$.\smallskip

Let us fix $x \notin [-1,1]$ and since $y$ tends to $1$ we can assume $x \neq y$ and $y \notin K$. We use monotone convergence twice and plug in the result of Lemma \ref{lemma_boring}:
\begin{align}\label{v_1_sep}
&\quad \mE^x\big[ \1_{\{ T_{K^\mathrm{C}}<T_{[-1,1]}\}} v_1(\xi_{T_{K^\mathrm{C}}}) \big]\nonumber \\
&= \lim_{\eps \searrow 0} \mE^x\big[ \1_{\{ \xi_{T_{K^\mathrm{C}}}>1+\eps \text{ or } \xi_{T_{K^\mathrm{C}}}<-1 \}} v_1(\xi_{T_{K^\mathrm{C}}}) \big]\nonumber  \\
&= \lim_{\eps \searrow 0} \lim_{y \searrow 1} \mE^x\big[ \1_{\{ \xi_{T_{K^\mathrm{C}}}>y+\eps \text{ or } \xi_{T_{K^\mathrm{C}}}<-1 \}} v_1(\xi_{T_{K^\mathrm{C}}})\big]\nonumber \\
&= \lim_{\eps \searrow 0}\lim_{y \searrow 1} \frac{2^{\alpha\hat\rho-1}c_{\alpha\rho}}{g(y)} \mE^x\big[ \1_{\{ \xi_{T_{K^\mathrm{C}}}>y+\eps \text{ or } \xi_{T_{K^\mathrm{C}}}<-1 \}} u_{[-1,1]}(\xi_{T_{K^\mathrm{C}}},y) \big] \nonumber \\
&\quad -  \lim_{\eps \searrow 0}\lim_{y \searrow 1}\mE^x \Big[\big(\sin(\pi\alpha\hat\rho)\1_{\{ \xi_{T_{K^\mathrm{C}}}>y+\eps \}}+ \sin(\pi\alpha\rho)\1_{\{\xi_{T_{K^\mathrm{C}}}<-1 \}}\big) \\ 
&\quad \quad \times \frac{(1-\alpha\hat\rho)|\xi_{T_{K^\mathrm{C}}}-y|^{\alpha-1}}{g(y)}\int\limits_1^{z(\xi_{T_{K^\mathrm{C}}},y)} (u-1)^{\alpha\rho} (u+1)^{\alpha\hat\rho-2} \, \dd u \Big] \nonumber\\
&\quad +  \lim_{\eps \searrow 0}\lim_{y \searrow 1}(\alpha-1)_+ \Big(\frac{\alpha\rho}{g(y)} \int\limits_1^y \psi_{\alpha\hat\rho}(u)\,\dd u -1 \Big) \nonumber \\
&\quad \quad \times \mE^x \Big[\sin(\pi\alpha\hat\rho)\1_{\{ \xi_{T_{K^\mathrm{C}}}>y+\eps \}} \int\limits_1^{\xi_{T_{K^\mathrm{C}}}} \psi_{\alpha\rho}(u)\, \dd u  + \sin(\pi\alpha\rho)\1_{\{\xi_{T_{K^\mathrm{C}}}<-1 \}}\int\limits_1^{|\xi_{T_{K^\mathrm{C}}}|} \psi_{\alpha\hat\rho}(u)\, \dd u\Big] \nonumber
\end{align}

We care about these three summands separately. We start with the last one which just appears if $\alpha>1$. From the proof of Corollary \ref{cor_limit} we already know that
$$\frac{\alpha\rho}{g(y)} \int\limits_1^y \psi_{\alpha\hat\rho}(u)\,\dd u -1$$
converges to $0$ for $y \searrow 0$. Furthermore, we get with monotone convergence:
\begin{align*}
&\quad\lim_{\eps \searrow 0}\lim_{y \searrow 1} \mE^x \Big[\sin(\pi\alpha\hat\rho)\1_{\{ \xi_{T_{K^\mathrm{C}}}>y+\eps \}} \int\limits_1^{\xi_{T_{K^\mathrm{C}}}} \psi_{\alpha\rho}(u)\, \dd u + \sin(\pi\alpha\rho)\1_{\{\xi_{T_{K^\mathrm{C}}}<-1 \}}\int\limits_1^{|\xi_{T_{K^\mathrm{C}}}|} \psi_{\alpha\hat\rho}(u)\, \dd u\Big] \\
&= \mE^x \Big[\sin(\pi\alpha\hat\rho)\1_{\{ \xi_{T_{K^\mathrm{C}}}>1 \}} \int\limits_1^{\xi_{T_{K^\mathrm{C}}}} \psi_{\alpha\rho}(u)\, \dd u + \sin(\pi\alpha\rho)\1_{\{\xi_{T_{K^\mathrm{C}}}<-1 \}}\int\limits_1^{|\xi_{T_{K^\mathrm{C}}}|} \psi_{\alpha\hat\rho}(u)\, \dd u\Big]\\
&= \frac{\pi}{\Gamma(1-\alpha\rho)\Gamma(1-\alpha\hat\rho)} \mE^x \Big[\1_{\{ T_{K^\mathrm{C}} < T_{[-1,1]} \}} h(\xi_{T_{K^\mathrm{C}}}) \Big],
\end{align*}
where $h$ is the invariant function which appears in Döring et al. \cite{Doer_Kyp_Wei_01}. But since the $h$-transformed process with this invariant function is transient with infinite lifetime (see Theorem 1.3 in that article) it leaves all compact sets almost surely. Hence, we have
\begin{align*}
\mE^x \Big[\1_{\{ T_{K^\mathrm{C}} < T_{[-1,1]} \}} \frac{h(\xi_{T_{K^\mathrm{C}}})}{h(x)} \Big] &= \mP^x_h(T_{K^\mathrm{C}} < \zeta) \\
&= \mP^x_h(T_{K^\mathrm{C}} < \infty) \\
&= 1,
\end{align*}
thus, $\mE^x \big[\1_{\{ T_{K^\mathrm{C}} < T_{[-1,1]} \}} h(\xi_{T_{K^\mathrm{C}}}) \big] = h(x) <\infty$. It follows that the third term of (\ref{v_1_sep}) is $0$. So it remains to consider the first and the second sumand of (\ref{v_1_sep}). With Proposition 6.2 (ii) of Kunita and Watanabe \cite{Kun_Wat_01} and Corollary \ref{cor_limit} we see for the first term:
\begin{align*}
&\quad\lim_{\eps \searrow 0}\lim_{y \searrow 1} \frac{2^{\alpha\hat\rho-1}c_{\alpha\rho}}{g(y)} \mE^x\big[ \1_{\{ T_{K^\mathrm{C}}<T_{[-1,1]}\}}\1_{\{ \xi_{T_{K^\mathrm{C}}}>y+\eps \text{ or } \xi_{T_{K^\mathrm{C}}}<-1 \}} u_{[-1,1]}(\xi_{T_{K^\mathrm{C}}},y) \big] \\
&= \lim_{y \searrow 1} \frac{2^{\alpha\hat\rho-1}c_{\alpha\rho}}{g(y)} \mE^x\big[ \1_{\{ T_{K^\mathrm{C}}<T_{[-1,1]}\}} u_{[-1,1]}(\xi_{T_{K^\mathrm{C}}},y) \big] \\
& \quad- \lim_{\eps \searrow 0}\lim_{y \searrow 1} \frac{2^{\alpha\hat\rho-1}c_{\alpha\rho}}{g(y)} \mE^x\big[ \1_{\{ T_{K^\mathrm{C}}<T_{[-1,1]}\}}\1_{\{ \xi_{T_{K^\mathrm{C}}} \in (1,y+\eps) \}} u_{[-1,1]}(\xi_{T_{K^\mathrm{C}}},y) \big] \\
&= \lim_{y \searrow 1} \frac{2^{\alpha\hat\rho-1}c_{\alpha\rho}}{g(y)} u_{[-1,1]}(x,y)\\
& \quad- \lim_{\eps \searrow 0}\lim_{y \searrow 1} \frac{2^{\alpha\hat\rho-1}c_{\alpha\rho}}{g(y)} \mE^x\big[ \1_{\{ T_{K^\mathrm{C}}<T_{[-1,1]}\}}\1_{\{ \xi_{T_{K^\mathrm{C}}} \in (1,y+\eps) \}} u_{[-1,1]}(\xi_{T_{K^\mathrm{C}}},y) \big] \\
&= v_1(x) - \lim_{\eps \searrow 0}\lim_{y \searrow 1} \frac{2^{\alpha\hat\rho-1}c_{\alpha\rho}}{g(y)} \mE^x \big[ \1_{\{ T_{K^\mathrm{C}}<T_{[-1,1]}\}}\1_{\{ \xi_{T_{K^\mathrm{C}}} \in (1,y+\eps) \}} u_{[-1,1]}(\xi_{T_{K^\mathrm{C}}},y) \big].
\end{align*}
Hence, to prove harmonicity of $v_1$ it sufficies to show
\begin{align}\label{beh_1}
\lim_{y \searrow 1}\mE^x&\Big[ \1_{\{ T_{K^\mathrm{C}}<T_{[-1,1]},|\xi_{T_{K^\mathrm{C}}}-y| >\eps \}}\frac{|\xi_{T_{K^\mathrm{C}}}-y|^{\alpha-1}}{g(y)}\int\limits_1^{z(\xi_{T_{K^\mathrm{C}}},y)} (u-1)^{\alpha\rho} (u+1)^{\alpha\hat\rho-2} \, \dd u \Big] =0
\end{align}
for all $\eps >0$ and
\begin{align}\label{beh_2}
\lim_{\eps \searrow 0} \lim_{y \searrow 1} \frac{1}{g(y)} \mE^x\big[ \1_{\{ T_{K^\mathrm{C}}<T_{[-1,1]}\}}\1_{\{ \xi_{T_{K^\mathrm{C}}} \in (1,y+\eps) \}} u_{[-1,1]}(\xi_{T_{K^\mathrm{C}}},y) \big]=0.
\end{align}

We start with (\ref{beh_1}). First we note that 
\begin{align}\label{est_z}
\begin{split}
\int\limits_1^{z(\xi_{T_{K^\mathrm{C}}},y)} (u-1)^{\alpha\rho} (u+1)^{\alpha\hat\rho-2} \, \dd u &\leq
(z(\xi_{T_{K^\mathrm{C}}},y)-1)^{\alpha\rho} \int\limits_1^{z(\xi_{T_{K^\mathrm{C}}},y)} (u+1)^{\alpha\hat\rho-2} \, \dd u \\
&\leq C_1 (z(\xi_{T_{K^\mathrm{C}}},y)-1)^{\alpha\rho}\\
&=\begin{cases}
C_1 \frac{(\xi_{T_{K^\mathrm{C}}}+1)^{\alpha\rho}(y-1)^{\alpha\rho}}{|\xi_{T_{K^\mathrm{C}}}-y|^{\alpha\rho}} \quad &\text{if } \xi_{T_{K^\mathrm{C}}}>y+\eps \\
C_1 \frac{(|\xi_{T_{K^\mathrm{C}}}|-1)^{\alpha\rho}(y-1)^{\alpha\rho}}{|\xi_{T_{K^\mathrm{C}}}-y|^{\alpha\rho}} \quad &\text{if } \xi_{T_{K^\mathrm{C}}}<-1
\end{cases}\\
&\leq C_1 \frac{(|\xi_{T_{K^\mathrm{C}}}|+1)^{\alpha\rho}(y-1)^{\alpha\rho}}{|\xi_{T_{K^\mathrm{C}}}-y|^{\alpha\rho}}
\end{split}
\end{align}
where $C_1 = \int\limits_1^{\infty} (u+1)^{\alpha\hat\rho-2} \, \dd u<\infty$.  With that we get on $\{ |\xi_{T_{K^\mathrm{C}}}-y| >\eps \}$ (without loss of generality we assume $y<2$):
\begin{align*}
&\quad \frac{|\xi_{T_{K^\mathrm{C}}}-y|^{\alpha-1}}{g(y)} \int\limits_1^{z(\xi_{T_{K^\mathrm{C}}},y)} (u-1)^{\alpha\rho} (u+1)^{\alpha\hat\rho-2} \, \dd u \\
&\leq  C_1 \frac{|\xi_{T_{K^\mathrm{C}}}-y|^{\alpha-1}}{g(y)} \frac{(|\xi_{T_{K^\mathrm{C}}}|+1)^{\alpha\rho}(y-1)^{\alpha\rho}}{|\xi_{T_{K^\mathrm{C}}}-y|^{\alpha\rho}}\\
&= C_1 |\xi_{T_{K^\mathrm{C}}}-y|^{\alpha\hat{\rho}-1} (|\xi_{T_{K^\mathrm{C}}}|+1)^{\alpha\rho}(y+1)^{1-\alpha\hat\rho}\\
&\leq C_1 |\xi_{T_{K^\mathrm{C}}}-y|^{\alpha\hat{\rho}-1} (|\xi_{T_{K^\mathrm{C}}}-y|^{\alpha\rho}+ (y+1)^{\alpha\rho}) (y+1)^{1-\alpha\hat\rho}\\
&= C_1 (y+1)^{1-\alpha\hat\rho}(|\xi_{T_{K^\mathrm{C}}}-y|^{\alpha-1}+ |\xi_{T_{K^\mathrm{C}}}-y|^{\alpha\hat{\rho}-1}(y+1)^{\alpha\rho})\\
&\leq C_1 3^{1-\alpha\hat\rho}(\eps^{\alpha-1}+ 3^{\alpha\rho}\eps^{\alpha\hat\rho-1})\\
&\leq C_1 3^{1+\alpha\rho-\alpha\hat\rho}(\eps^{\alpha-1}+ \eps^{\alpha\hat\rho-1}) =: C_\eps
\end{align*}

Hence, we can use dominated convergence to switch the $y$-limit and the expectation in \eqref{beh_1}. The following calculation on $\{ |\xi_{T_{K^\mathrm{C}}}-y| >\eps \}$ shows that the integrand converges pointwise to $0$ which shows \eqref{beh_1}:
\begin{align*}
&\quad\frac{|\xi_{T_{K^\mathrm{C}}}-y|^{\alpha-1}}{g(y)}\int\limits_1^{z(\xi_{T_{K^\mathrm{C}}},y)} (u-1)^{\alpha\rho} (u+1)^{\alpha\hat\rho-2} \, \dd u \\
&\leq 2^{\alpha\hat\rho-2} \frac{|\xi_{T_{K^\mathrm{C}}}-y|^{\alpha-1}}{g(y)}\int\limits_1^{z(\xi_{T_{K^\mathrm{C}}},y)} (u-1)^{\alpha\rho} \, \dd u \\
&= \frac{2^{\alpha\hat\rho-2}}{\alpha\rho+1} \frac{|\xi_{T_{K^\mathrm{C}}}-y|^{\alpha-1}}{g(y)}(z(\xi_{T_{K^\mathrm{C}}},y)-1)^{\alpha\rho+1}\\
&\leq \frac{2^{\alpha\hat\rho-2}}{\alpha\rho+1} \frac{|\xi_{T_{K^\mathrm{C}}}-y|^{\alpha-1}}{{(y-1)^{\alpha\rho}(y+1)^{\alpha\hat\rho-1}}} \frac{(|\xi_{T_{K^\mathrm{C}}}|+1)^{\alpha\rho+1}(y-1)^{\alpha\rho+1}}{|\xi_{T_{K^\mathrm{C}}}-y|^{\alpha\rho+1}}\\
&= \frac{2^{\alpha\hat\rho-2}}{\alpha\rho+1} (|\xi_{T_{K^\mathrm{C}}}-y|^{\alpha\hat\rho-2}(|\xi_{T_{K^\mathrm{C}}}|+1)^{\alpha\rho+1}(y-1)(y+1)^{1-\alpha\hat\rho}\\
&\overset{y \searrow 1}{\longrightarrow} 0 
\end{align*}
where we used the same estimate for $z(\xi_{T_{K^\mathrm{C}}},y)-1$ as in \eqref{est_z}. This shows (\ref{beh_1}).\smallskip

Now we show (\ref{beh_2}). We define $a = \min(\inf(K \cap (1,\infty),-\sup(K \cap (-\infty,-1))$. Sinc $y$ tends to $1$ and $\eps$ to $0$ we can assume $a>y+\eps$. It follows that $\xi_{T_{K^\mathrm{C}}} \in (1,y+\eps)$ is just possible if $T_{K^\mathrm{C}} = T_{(-a,a)}$. So we have
\begin{align*}
&\quad \mE^x\big[ \1_{\{ T_{K^\mathrm{C}}<T_{[-1,1]}\}}\1_{\{ \xi_{T_{K^\mathrm{C}}} \in (1,y+\eps) \}} u_{[-1,1]}(\xi_{T_{K^\mathrm{C}}},y) \big]\\
&\leq \mE^x\big[ \1_{\{ \xi_{T_{(-a,a)}}\in (1,y+\eps),T_{(-a,a)}<\infty \}} u_{[-1,1]}(\xi_{T_{(-a,a)}},y) \big].
\end{align*}
Further $\xi_{T_{(-a,a)}} =y $ happens with zero probability and with this follows
\begin{align*}
&\quad\mE^x\big[ \1_{\{ \xi_{T_{(-a,a)}} \in (1,y+\eps),T_{(-a,a)}<\infty \}} u_{[-1,1]}( \xi_{T_{(-a,a)}},y) \big]\\
&=\mE^x\big[ \1_{\{ \xi_{T_{(-a,a)}} \in (1,y),T_{(-a,a)}<\infty \}} u_{[-1,1]}(\xi_{T_{(-a,a)}},y) \big]\\
&\quad+\mE^x\big[ \1_{\{ \xi_{T_{(-a,a)}} \in (y,\eps),T_{(-a,a)}<\infty \}} u_{[-1,1]}( \xi_{T_{(-a,a)}},y) \big].
\end{align*}
With the formulas for $u_{[-1,1]}$ of Profeta and Simon \cite{Pro_Sim_01} we get for $\xi_{T_{(-a,a)}} \in (1,y)$:
\begin{align*}
u_{[-1,1]}(\xi_{T_{(-a,a)}},y) &\leq \frac{2^{1-\alpha}}{\Gamma(\alpha\rho)\Gamma(\alpha\hat\rho)} (y-\xi_{T_{(-a,a)}})^{\alpha-1} \int\limits_1^{z(\xi_{T_{(-a,a)}},y)} (u-1)^{\alpha\hat\rho-1}(u+1)^{\alpha\rho-1}\, \dd u\\
&\leq \frac{2^{-\alpha\hat\rho}}{\alpha\hat\rho\Gamma(\alpha\rho)\Gamma(\alpha\hat\rho)} (y-\xi_{T_{(-a,a)}})^{\alpha-1} (z(\xi_{T_{(-a,a)}},y)-1)^{\alpha\hat\rho}\\
&\leq \frac{2^{-\alpha\hat\rho}}{\alpha\hat\rho\Gamma(\alpha\rho)\Gamma(\alpha\hat\rho)} (y-\xi_{T_{(-a,a)}})^{\alpha-1} \Big(\frac{(\xi_{T_{(-a,a)}}-1)(y+1)}{y-\xi_{T_{(-a,a)}}}\Big)^{\alpha\hat\rho}\\
&= \frac{1}{\alpha\hat\rho\Gamma(\alpha\rho)\Gamma(\alpha\hat\rho)} (y-\xi_{T_{(-a,a)}})^{\alpha\rho-1} (\xi_{T_{(-a,a)}}-1)^{\alpha\hat\rho}\\
&\leq \frac{1}{\alpha\hat\rho\Gamma(\alpha\rho)\Gamma(\alpha\hat\rho)} (y-\xi_{T_{(-a,a)}})^{\alpha\rho-1} (y-1)^{\alpha\hat\rho}.
\end{align*}
It follows for $x>a$ with Theorem 1.1 of Kyprianou et al. \cite{Kyp_Par_Wat_01} and the scaling property:
\begin{align*}
&\quad\mE^x \la \1_{\lc \xi_{T_{(-a,a)}} \in (1,y),T_{(-a,a)}<\infty \rc} u_{[-1,1]}(\xi_{T_{(-a,a)}},y) \ra \\
&\leq \frac{1}{\Gamma(\alpha\rho)\Gamma(\alpha\hat\rho)} (y-1)^{\alpha\hat\rho}\mE^x \la \1_{\lc \xi_{T_{(-a,a)}} \in (1,y), T_{(-a,a)}<\infty\rc}(y-\xi_{T_{(-a,a)}})^{\alpha\rho-1} \ra \\
&\leq \frac{\sin(\pi\alpha\hat\rho)}{\pi \Gamma(\alpha\rho)\Gamma(\alpha\hat\rho)} (x+a)^{\alpha\rho}(x-a)^{\alpha\hat\rho} (y-1)^{\alpha\hat\rho} \int\limits_{(1,y)} \frac{(y-u)^{\alpha\rho-1}}{(a+u)^{\alpha\rho}(a-u)^{\alpha\hat\rho} (x-u)}\, \dd u\\
&\leq \frac{a\sin(\pi\alpha\hat\rho)}{\pi \Gamma(\alpha\rho)\Gamma(\alpha\hat\rho)} \frac{(x+a)^{\alpha\rho}(x-a)^{\alpha\hat\rho}}{(a+1)^{\alpha\rho}(a-y)^{\alpha\hat\rho} (x-y)} (y-1)^{\alpha\hat\rho} \int\limits_{(1,y)} (y-u)^{\alpha\rho-1} \dd u\\
&= \frac{a\sin(\pi\alpha\hat\rho)}{\pi \alpha\rho \Gamma(\alpha\rho)\Gamma(\alpha\hat\rho)} \frac{(x+a)^{\alpha\rho}(x-a)^{\alpha\hat\rho}}{(a+1)^{\alpha\rho}(a-y)^{\alpha\hat\rho} (x-y)} (y-1)^{\alpha\hat\rho} (y-1)^{\alpha\rho}.
\end{align*}
With this estimate we see immediately 
$$\lim_{y \searrow 1} \frac{1}{(y-1)^{\alpha\rho}} \mE^x \big[ \1_{\{ \xi_{T_{(-a,a)}} \in (1,y) \}} u_{[-1,1]}(\xi_{T_{(-a,a)}},y) \big] = 0$$
for $x>1$. For $x<-1$ we use Theorem 1.1 of Kyprianou et al. \cite{Kyp_Par_Wat_01} in a similar way to deduce the analogous claim. Similarly we get for $\xi_{T_{(-a,a)}} \in (y,y+\eps)$ (without loss of generality $y+\eps<2$):
\begin{align*}
u_{[-1,1]}(\xi_{T_{(-a,a)}},y) &\leq \frac{2^{1-\alpha}}{\Gamma(\alpha\rho)\Gamma(\alpha\hat\rho)}(\xi_{T_{(-a,a)}}-y)^{\alpha-1} \int\limits_1^{z(\xi_{T_{K^\mathrm{C}}},y)} (u-1)^{\alpha\rho-1} (u+1)^{\alpha\hat\rho-1} \, \dd u\\
&\leq \frac{2^{-\alpha\rho}}{\alpha\rho\Gamma(\alpha\rho)\Gamma(\alpha\hat\rho)} (\xi_{T_{(-a,a)}}-y)^{\alpha-1} (z(\xi_{T_{(-a,a)}},y)-1)^{\alpha\rho} \\
&= \frac{2^{-\alpha\rho}}{\alpha\rho\Gamma(\alpha\rho)\Gamma(\alpha\hat\rho)} (\xi_{T_{(-a,a)}}-y)^{\alpha\hat\rho-1} (\xi_{T_{(-a,a)}}+1)^{\alpha\rho}(y-1)^{\alpha\rho}\\
&\leq \frac{2^{-\alpha\rho} 3^{\alpha\rho}}{\alpha\rho\Gamma(\alpha\rho)\Gamma(\alpha\hat\rho)}(y-1)^{\alpha\rho} (\xi_{T_{(-a,a)}}-y)^{\alpha\hat\rho-1}.
\end{align*}
Define $C_2 := \frac{2^{-\alpha\rho} 3^{\alpha\rho}}{\alpha\rho\Gamma(\alpha\rho)\Gamma(\alpha\hat\rho)}$ and we get again with Theorem 1.1 of \cite{Kyp_Par_Wat_01} for $x>1$
\begin{align*}
&\quad\mE^x \big[ \1_{\{ \xi_{T_{(-a,a)}} \in (y,y+\eps),T_{(-a,a)}<\infty \}} u_{[-1,1]}(\xi_{T_{(-a,a)}},y) \big] \\
&\leq C_2(y-1)^{\alpha\rho} \mE^x \big[ \1_{\{ \xi_{T_{(-a,a)}} \in (y,y+\eps),T_{(-a,a)}<\infty \}} (\xi_{T_{(-a,a)}}-y)^{\alpha\hat\rho-1} \big] \\
&\leq \frac{C_2a\sin(\pi\alpha\hat\rho)}{\pi} (x+a)^{\alpha\rho}(x-a)^{\alpha\hat\rho} (y-1)^{\alpha\rho} \int\limits_{(y,y+\eps)} \frac{(u-y)^{\alpha\hat\rho-1}}{(a+u)^{\alpha\rho}(a-u)^{\alpha\hat\rho} (x-u)}\, \dd u \\
&\leq \frac{C_2a\sin(\pi\alpha\hat\rho)}{\pi} \frac{(x+a)^{\alpha\rho}(x-a)^{\alpha\hat\rho}(y-1)^{\alpha\rho}}{(a+y)^{\alpha\rho}(a-(y+\eps))^{\alpha\hat\rho} (x-(y+\eps))} \int\limits_{(y,y+\eps)} (u-y)^{\alpha\hat\rho-1} \dd u\\
&= \frac{C_2a\sin(\pi\alpha\hat\rho)}{\pi} \frac{(x+a)^{\alpha\rho}(x-a)^{\alpha\hat\rho}}{(a+y)^{\alpha\rho}(a-(y+\eps))^{\alpha\hat\rho} (x-(y+\eps))} \frac{(y-1)^{\alpha\rho} \eps^{\alpha\hat\rho}}{\alpha\hat\rho}.
\end{align*}
So we have:
\begin{align*}
&\quad\lim_{\eps \searrow 0}\lim_{y \searrow 1} \frac{1}{(y-1)^{\alpha\rho}} \mE^x\big[ \1_{\lc T_{K^\mathrm{C}}<T_{[-1,1]}\rc}\1_{\lc \xi_{T_{K^\mathrm{C}}} \in (1,y+\eps) \rc} u_{[-1,1]}(\xi_{T_{K^\mathrm{C}}},y) \big]\\
&\leq \frac{C_2\sin(\pi\alpha\hat\rho)}{\pi\alpha\hat\rho} \lim_{\eps \searrow 0}\lim_{y \searrow 1} \Big[\frac{(x+a)^{\alpha\rho}(x-a)^{\alpha\hat\rho}}{(a+y)^{\alpha\rho}(a-(y+\eps))^{\alpha\hat\rho} (x-(y+\eps))} \eps^{\alpha\hat\rho} \Big]\\
&= 0.
\end{align*}
The claim for $x<-1$ follows again similarly. This shows (\ref{beh_2}) and hence, we have harmonicity of $v_1$.
\end{proof}

\begin{remark}
If $\alpha \leq 1$, another (maybe more elegant) way of proving harmonicity of $v_1$ is to prove that the renewal densities of the MAP which corresponds to the stable process via the Lamperti-Kiu transform (for explicit expressions see Corollary 1.6 of \cite{Kyp_Riv_Sen_01}) are harmonic functions for the MAP killed on entering the negative half-line. This claim should be true since Silverstein \cite{Sil_01} proved the analogous claim for a Lévy process which does not drift to $-\infty$. One can show that $v_1$ and $v_{-1}$ are just these renewal densities (the argument replaced by the logarithm). Via the Lamperti-Kiu transform one could obtain harmonicity of $v_1$ and $v_{-1}$ for the stable process killed in $[-1,1]$.  
\end{remark}

\subsection{Behaviour at the killing time}\label{sec_proof_paths}
Before we start with the proofs we should discuss more elementary properties of $v_1$ and $v_{-1}$. First, it can be seen immediately that $v_1$ has a pole in $1$ and $v_{-1}$ has a pole in $-1$ and hence $v:= v_1 + v_{-1}$ has poles in $1$ and $-1$. Further $v_1$ is bounded on $(-\infty,-1) \cup (K,\infty)$ for all $K>1$. For $\alpha \leq 1$ this is obvious and for $\alpha>1$ this can be seen via showing that $v_1$ converges for $x \rightarrow \pm \infty$ (a similar convergence was shown in \cite{Doer_Kyp_Wei_01} in the proof of Lemma 3.3). Similarly $v_{-1}$ is bounded on $(-\infty,-K) \cup (1,\infty)$ for all $K>1$. It follows obviously that $v$ is bounded on $(-\infty,-K_1) \cup (K_2,\infty)$ for all $K_1,K_2>1$.\smallskip

For the first results we need to define the potential of the $h$-transformed process via
$$U_{v_1}(x,\dd y) = \mE^x_{v_1} \Big[ \int\limits_0^\zeta \1_{\lc \xi_t \in \dd y \rc} \, \dd t \Big],\quad x,y \notin [-1,1],$$
which is the expected time the process $(\xi,\mP^x_{v_1})$ stays in $\dd y$ until it is killed. With a Fubini flip we obtain
$$U_{v_1}(x,\dd y) = \frac{v_1(y)}{v_1(x)}\,U_{[-1,1]}(x,\dd y) = \frac{v_1(y)}{v_1(x)}u_{[-1,1]}(x,y)\, \dd y.$$
The following result shows on the one hand that the $h$-transformed process is almost surely bounded and second that the expected time the process stays in a set of the form $[-b,-1)\cup(1,b]$ is finite.
\begin{lemma}\label{lemma_help_v1}
Let $\xi$ be an $\alpha$-stable process with $\alpha \in (0,2)$ and both sided jumps. Then it holds for $x\notin [-1,1]$:
\begin{enumerate}[(i)]
\item $\mP_{v_1}^x(T_{(-\infty,-d]\cup[d,\infty)} < \zeta \, \forall d>1)=0$.
\item $U_{v_1}(x,[-b,-1)\cup(1,b]) <\infty$ for all $b>1$.
\end{enumerate}
\end{lemma}

\begin{proof}
(i) We already noticed that $v_1$ is bounded on $(-\infty,-K)\cup(K,\infty)$ for all $K>1$. So we obtain, applying dominated convergence in the last equality,
\begin{align*}
\mP_{v_1}^x(T_{(-\infty,-d]\cup[d,\infty)} <\zeta \, \forall d>1) &= \lim_{d \rightarrow \infty}  \mP_{v_1}^x(T_{(-\infty,-d]\cup[d,\infty)} <\zeta)\\
&= \lim_{d \rightarrow \infty} \mE^x \Big[ \1_{\{ T_{(-\infty,-d]\cup[d,\infty)}<T_{[-1,1]} \}} \frac{v_1(\xi_{T_{(-\infty,-d]\cup[d,\infty)}})}{v_1(x)} \Big] \\
&= \mE^x \Big[ \lim_{d \rightarrow \infty}  \1_{\{ T_{(-\infty,-d]\cup[d,\infty)}<T_{[-1,1]}\}} \frac{v_1(\xi_{T_{(-\infty,-d]\cup[d,\infty)}})}{v_1(x)} \Big].
\end{align*}
In the case $\alpha<1$ we use that $v_1(y)$ converges to $0$ for $y \rightarrow \pm \infty$. If $\alpha \geq 1$ we see that $\1_{\{ T_{(-\infty,-d]\cup[d,\infty)}<T_{[-1,1]}\}}$ converges to $0$ almost surely since $(\xi,\mP^x)$ is recurrent. This shows (i).\smallskip

(ii) It holds
$$U_{v_1}(x,[-b,-1)\cup(1,b]) = \frac{1}{v_1(x)}\int\limits_{[-b,-1)\cup(1,b]} v_1(y) u_{[-1,1]}(x,y) \, \dd y.$$
Since $v_1$ is bounded and $u_{[-1,1]}(x,\cdot)$ is integrable on all compact intervals, the only points where this integral could be infinite, are the boundary points $1$ and $-1$. From the explicit formulas of \cite{Pro_Sim_01} we see that $u_{[-1,1]}(x,y)$ converges to $0$ for $y \rightarrow \pm 1$. Further $v_1(y)$ behaves as $(y-1)^{\alpha\hat{\rho}-1}$ for $y \searrow 1$ and as $(|y|-1)^{\alpha\rho}$ for $y \nearrow -1$. Since $\alpha\rho,\alpha\hat\rho \in (0,1)$ these arguments shows $U_{v_1}(x,[-b,-1)\cup(1,b])<\infty$.
\end{proof}

Combining the two statements of Lemma \ref{lemma_help_v1} we can show Proposition \ref{thm_abs_above}.
\begin{proof}[Proof of Proposition \ref{thm_abs_above}]
We show that $\mP_{v_1}^x(\zeta < \infty)=1$ and $\mP_{v_1}^x(\xi_{\zeta-}=1)=1$ and start with the first equality. From Lemma \ref{lemma_help_v1} (ii) we know
$$\mP_{v_1}^x\Big(\int\limits_0^\zeta \1_{\lc \xi_t \in [-b,-1)\cup(1,b] \rc} \, \dd t <\infty \Big)=1$$
for all $b>1$. By the continuity of probability measures we see
\begin{align*}
&\quad\mP_{v_1}^x\Big(\int\limits_0^\zeta \1_{\{ \xi_t \in [-b,-1)\cup(1,b] \}} \, \dd t <\infty  \, \forall b>1\Big)\\
&= \lim\limits_{b \rightarrow \infty} \mP_{v_1}^x \Big(\int\limits_0^\zeta \1_{\{ \xi_t \in [-b,-1)\cup(1,b] \}} \, \dd t <\infty \Big) = 1.
\end{align*}
On the other hand Lemma \ref{lemma_help_v1} (i) yields
$$\mP_{v_1}^x(\exists d>1: \, T_{(-\infty,-d]\cup [d,\infty)} \geq \zeta)=1.$$
Since the intersection of two events with probability $1$ has again probability $1$ it follows:
\begin{align*}
\mP_{v_1}^x(\zeta < \infty) &= \mP_{v_1}^x \Big(\int\limits_0^\zeta \1_{\{ \xi_t \in \mR \setminus [-1,1] \}} \, \dd t <\infty \Big)\\
&\geq \mP_{v_1}^x \Big(\int\limits_0^\zeta \1_{\{ \xi_t \in [-b,-1)\cup(1,b] \}} \, \dd t <\infty \, \forall b>1  \, , \, \exists d>1: T_{(-\infty,-d]\cup[d,\infty)} \geq \zeta \Big)\\
&= 1.
\end{align*}

To prove $\mP_{v_1}^x(\xi_{\zeta-}=1)=1$ we use a procedure which is inspired by Chaumont \cite{Chau_01}. Using that $v_1$ is harmonic (Theorem \ref{thm_harm}) we see for $x \notin [-1,1]$ and
$$M_{a,b} = (-\infty,-b) \cup (-a,-1) \cup (1,1+\eps) \cup (b,\infty)$$
with $1<a<b$ and $\eps>0$ (obviously the complement of $M_{a,b}$ is compact in $\mR\setminus [-1,1]$):
$$\mE^x \big[ \1_{\{ T_{M_{a,b}} < T_{[-1,1]} \}} v_1(\xi_{T_{M_{a,b}}}) \big] = v_1(x).$$
It follows that
$$\mP_{v_1}^x(T_{M_{a,b}} < \zeta) = \frac{1}{v_1(x)}\mE^x \big[ \1_{\{ T_{M_{a,b}} < T_{[-1,1]} \}} v_1(\xi_{T_{M_{a,b}}}) \big] = 1.$$
From Lemma \ref{lemma_help_v1} we know on the one hand
\begin{align}\label{help_11}
\mP_{v_1}^x(T_{(-\infty,-b) \cup (b,\infty)} < \zeta \, \forall \, b>1)&= 0.
\end{align}
On the other hand we see, applying dominated convergence using that $v_1$ is bounded on $(-\infty,-1)$, 
\begin{align}\label{help_12}
\begin{split}
\mP_{v_1}^x(T_{(-a,-1)} < \zeta \, \forall \, a>1)&= \lim_{a \searrow 1} \mP_{v_1}^x(T_{(-a,-1)} < \zeta) \\
&= \lim_{a \searrow 1} \frac{1}{v_1(x)}\mE^x \big[ \1_{\{ T_{(-a,-1)} < T_{[-1,1]} \}} v_1(\xi_{T_{(-a,-1)}}) \big] \\
&= \frac{1}{v_1(x)} \mE^x \big[ \lim_{a \searrow 1} \1_{\{ T_{(-a,-1)} < T_{[-1,1]} \}} v_1(\xi_{T_{(-a,-1)}}) \big] \\
&= 0.
\end{split}
\end{align}
In the last step we used that $v_1(y)$ converges to $0$ for $y \nearrow -1$. Note that this argument does not work if $(-a,-1)$ is replaced by $(1,a)$ because $v_1$ has a pole in $1$. Now we plug in \eqref{help_11} and \eqref{help_12} to obtain for all $\eps>0$:
\begin{align*}
&\quad \mP_{v_1}^x(T_{(1,1+\eps)} < \zeta)\\
&= \mP_{v_1}^x(\{ T_{(1,1+\eps)} < \zeta \} \cup \{ T_{(-a,-1)} < \zeta \, \forall \, a>1 \} \cup \{ T_{(-\infty,-b) \cup (b,\infty)} < \zeta \, \forall \, b>1 \}) \\
&= \lim_{b \rightarrow \infty} \lim_{a \searrow 1} \mP_{v_1}^x(\lc T_{(1,1+\eps)} < \zeta \rc \cup \lc T_{(-a,-1)} < \zeta \rc \cup \lc T_{(-\infty,-b ) \cup (b,\infty)} < \zeta \rc) \\
&= \lim_{b \rightarrow \infty}  \lim_{a \searrow 1} \mP_{v_1}^x( T_{M_{a,b}} < \zeta) \\
&= 1.
\end{align*}
With this in hand we show the final claim that $\xi_{\zeta-}=1$ almost surely under $\mP^x_{v_1}$.  By $(1,1+\delta)^{\mathrm{C}}$ we mean as usual $\mR\setminus [-1,1] \setminus (1,1+\delta)$.
\begin{align}\label{leftlimit}
\begin{split}
\mP^x_{v_1}(\xi_{\zeta-}=1) &= \mP^x_{v_1}(\forall \,\delta>0 \, \exists\, \eps \in (0,\delta] \,:\, \xi_t \in (1,1+\delta) \,\forall \, t \in [T_{(1,1+\eps)},\zeta) )\\
&= \lim_{\delta \searrow 0} \lim_{\eps \searrow 0} \mP^x_{v_1}(\xi_t \in (1,1+\delta) \,\forall \, t \in [T_{(1,1+\eps)},\zeta) ) \\
&=  \lim_{\delta \searrow 0} \lim_{\eps \searrow 0} \mE^x_{v_1}\big[ \mP_{v_1}^{\xi_{T_{(1,1+\eps)}}}(\xi_t \in (1,1+\delta) \,\forall \, t \in [0,\zeta)) \big] \\
&=  \lim_{\delta \searrow 0} \lim_{\eps \searrow 0} \mE^x_{v_1}\big[ \mP_{v_1}^{\xi_{T_{(1,1+\eps)}}}(T_{(1,1+\delta)^{\mathrm{C}}} \geq \zeta) \big] \\
&= 1- \lim_{\delta \searrow 0} \lim_{\eps \searrow 0} \mE^x_{v_1}\big[ \mP_{v_1}^{\xi_{T_{(1,1+\eps)}}}(T_{(1,1+\delta)^{\mathrm{C}}} < \zeta) \big] \\
&= 1- \lim_{\delta \searrow 0}  \mE^x_{v_1}\big[ \lim_{\eps \searrow 0}\mP_{v_1}^{\xi_{T_{(1,1+\eps)}}}(T_{(1,1+\delta)^{\mathrm{C}}} < \zeta) \big] \\
&= 1- \lim_{\delta \searrow 0}  \mE^x_{v_1}\big[ \lim_{\eps \searrow 0}\mP_{v_1}^{1+\eps}(T_{(1,1+\delta)^{\mathrm{C}}} < \zeta) \big].
\end{split}
\end{align}
In the second equality we used that $T_{(1,1+\eps)}<\zeta$ almost surely and in the third equality we used the strong Markov property of $(\xi,\mP^x_{v_1})$. Let us consider the $\eps$-limit inside the expectation. Using the definition of $\mP^x_{v_1}$ we see:
\begin{align*}
\mP_{v_1}^{1+\eps}(T_{(1,1+\delta)} < \zeta) &= \frac{1}{v_1(1+\eps)} \mE^x \big[ \1_{\{T_{(1,1+\delta)^{\mathrm{C}}} < T_{[-1,1]} \}} v_1(\xi_{T_{(1,1+\delta)^{\mathrm{C}}}}) \big].
\end{align*}
Since for fixed $\delta>0$ the function $v_1$ is bounded on $(-\infty,-1)\cup (1+\delta,\infty)$ and $\lim_{\eps \searrow 0}v_1(1+\eps) = \infty$ it follows that
$$\lim_{\eps \searrow 0}\mP_{v_1}^{1+\eps}(T_{(1,1+\delta)} < \zeta) = 0$$
and with \eqref{leftlimit} we conclude $\mP^x_{v_1}(\xi_{\zeta-}=1)=1$.

\end{proof}

Proposition \ref{thm_abs_both} can be proved similarly to Proposition \ref{thm_abs_above} using the following lemma:
\begin{lemma}\label{lemma_help_v}
Let $\xi$ be an $\alpha$-stable process with $\alpha \in (0,2)$ and both sided jumps. Then it holds:
\begin{enumerate}[(i)]
\item $\mP_{v}^x(T_{(-\infty,-d]\cup[d,\infty)} < \zeta \, \forall d>1)=0$.
\item $U_{v}(x,[-b,-1)\cup(1,b]) <\infty$ for all $b>1$.
\end{enumerate}
\end{lemma}
\begin{proof}
The proof is analogous to the one of Lemma \ref{lemma_help_v1}.
\end{proof}

The proof of Proposition \ref{thm_abs_both} consists of combining these two statements as in the proof of Proposition \ref{thm_abs_above}.

\subsection{Conditioning and $h$-transform} \label{sec_proof_conditioning}
To connect the $h$-transform with the conditioned process we need some connection between the harmonic function and the asymptotic probability of the event we condition on. We have to separate the cases $\alpha<1$ and $\alpha \geq 1$.
\subsubsection{The case $\alpha<1$}
\begin{proposition} \label{prop_as_<1}
Let $\xi$ be an $\alpha$-stable process with $\alpha \in (0,1)$ and both sided jumps. Then it holds:
\begin{align}\label{eq_as1_<1}
\frac{\pi \Gamma(1-\alpha\rho)\Gamma(1-\alpha\hat\rho)}{2^{\alpha}\Gamma(1-\alpha)} v_1(x) &= \lim_{\eps \searrow 0} \frac{1}{\eps}\mP^x(\xi_{\underline{m}} \in (1,1+\eps)), \quad x \in \mR \setminus [-1,1],
\end{align}
and
\begin{align}\label{eq_as_<1}
\frac{\pi \Gamma(1-\alpha\rho)\Gamma(1-\alpha\hat\rho)}{2^{\alpha}\Gamma(1-\alpha)} v(x) &= \lim_{\eps \searrow 0} \frac{1}{\eps}\mP^x(|\xi_{\underline{m}}| \in (1,1+\eps)), \quad x \in \mR \setminus [-1,1].
\end{align}
\end{proposition}

\begin{proof}
The proof is based on Proposition 1.1 of \cite{Kyp_Riv_Sen_01} where we find an explicit expression for the distribution of $\xi_{\underline{m}}$. For $x>1$ this gives
\begin{align*}
\mP^{x}(\xi_{\underline{m}} \in (1,1+\eps)) &= \frac{2^{-\alpha}\Gamma(1-\alpha\rho)}{\Gamma(1-\alpha)\Gamma(\alpha\hat\rho)} \int\limits_1^{x \land (1+\eps)} z^{-\alpha} (x-z)^{\alpha\hat{\rho}-1}(x+z)^{\alpha\rho} \, \dd z \\
&= \frac{2^{-\alpha}\Gamma(1-\alpha\rho)}{\Gamma(1-\alpha)\Gamma(\alpha\hat\rho)} \int\limits_1^{x \land (1+\eps)} z^{-1} \big(\frac{x}{z}-1 \big)^{\alpha\hat{\rho}-1} \big(\frac{x}{z}+1\big)^{\alpha\rho} \, \dd z \\
&= \frac{2^{-\alpha}\Gamma(1-\alpha\rho)}{\Gamma(1-\alpha)\Gamma(\alpha\hat\rho)} \int\limits_{\frac{x}{x \land (1+\eps)}}^{x} \frac{x}{z^2} \frac{z}{x} (z-1)^{\alpha\hat{\rho}-1}(z+1)^{\alpha\rho} \, \dd z \\
&= \frac{2^{-\alpha}\Gamma(1-\alpha\rho)}{\Gamma(1-\alpha)\Gamma(\alpha\hat\rho)} \int\limits_{1 \lor \frac{x}{1+\eps}}^{x}(1+\frac{1}{z}) \psi_{\alpha\rho}(z) \, \dd z.
\end{align*}
Applying l'Hopital's rule to the first calculation we obtain:
\begin{align*}
\frac{\Gamma(1-\alpha)\Gamma(\alpha\hat\rho)}{2^{-\alpha}\Gamma(1-\alpha\rho)} \lim_{\eps \searrow 0}\frac{1}{\eps} \mP^x(\xi_{\underline{m}} \in (1,1+\eps)) &= \lim_{\eps \searrow 0}\frac{1}{\eps} \int\limits_{\frac{x}{1+\eps}}^{x}\big(1+\frac{1}{z}\big) \psi_{\alpha\rho}(z) \, \dd z \\
&= \lim_{\eps \searrow 0} \frac{x}{(1+\eps)^2} \big(1+\frac{1+\eps}{x}\big) \psi_{\alpha\rho}(\frac{x}{1+\eps})\\
&=(x+1) \psi_{\alpha\rho}(x) \\
&= \frac{1}{\sin(\pi\alpha\hat{\rho})} v_1(x).
\end{align*}
Since $\sin(\pi\alpha\hat{\rho}) = \pi /(\Gamma(\alpha\hat{\rho})\Gamma(1-\alpha\hat{\rho}))$ this shows \eqref{eq_as1_<1} for $x>1$. For $x<-1$ we first use duality to deduce:
\begin{align*}
\mP^{x}(\xi_{\underline{m}} \in (1,1+\eps)) &= \hat{\mP}^{-x}(\xi_{\underline{m}} \in (-1-\eps,-1))\\
&= \frac{2^{-\alpha}\Gamma(1-\alpha\hat\rho)}{\Gamma(1-\alpha)\Gamma(\alpha\rho)} \int\limits_{1 \lor \frac{-x}{1+\eps}}^{-x}\big(1-\frac{1}{z}\big) \psi_{\alpha\hat\rho}(z) \, \dd z,
\end{align*}
where the second equality is verified using a similar calculation as above. Hence, it follows, for $x<-1$, that
\begin{align*}
\frac{\Gamma(1-\alpha)\Gamma(\alpha\rho)}{2^{-\alpha}\Gamma(1-\alpha\hat\rho)} \lim_{\eps \searrow 0} \frac{1}{\eps} \mP^x(\xi_{\underline{m}} \in (1,1+\eps)) &= \lim_{\eps \searrow 0} \frac{1}{\eps} \int\limits_{\frac{-x}{1+\eps}}^{-x}(1-\frac{1}{z}) \psi_{\alpha\hat\rho}(z) \, \dd z\\
&=\lim_{\eps \searrow 0} \frac{-x}{(1+\eps)^2} \big(1-\frac{1+\eps}{-x}\big) \psi_{\alpha\hat\rho}\big(\frac{-x}{1+\eps}\big) \\
&= (-x-1)\psi_{\alpha\hat\rho}(-x)\\
&= \frac{1}{\sin(\pi\alpha\rho)} v_1(x).
\end{align*}
Again we use $\sin(\pi\alpha\rho) = \frac{\pi}{\Gamma(\alpha\rho)\Gamma(1-\alpha\rho)}$ to obtain \eqref{eq_as1_<1} for $x<-1$.\smallskip

Similarly, \eqref{eq_as_<1} can be deduced as follows. Analogously to the proof of the first equation we can show
\begin{align*}
\frac{\pi \Gamma(1-\alpha\rho)\Gamma(1-\alpha\hat\rho)}{2^{\alpha}\Gamma(1-\alpha)} v_{-1}(x) &= \lim_{\eps \searrow 0} \frac{1}{\eps}\mP^x(\xi_{\underline{m}} \in (-(1+\eps)),-1)), \quad x \in \mR \setminus [-1,1].
\end{align*}
Since we defined $v(x) = v_1(x) + v_{-1}(x)$ this shows \eqref{eq_as_<1}.
\end{proof}

Now we are ready to prove the connection between the $h$-transform and the conditioned process.
\begin{proof}[Proof of Theorems \ref{thm_cond_v1_<1} and \ref{thm_cond_v_<1}]
We start with $x>1$. First note for $\delta>\eps>0$: 
\begin{align*}
&\quad\mP^x(\Lambda, t< T_{(-(1+\delta),1+\delta)} ,\xi_{\underline m} \in (1,1+\eps)) \\
&=\mP^x(\Lambda, t< T_{(-(1+\delta),1+\delta)}, t < \underline m ,\xi_{\underline m} \in (1,1+\eps)).
\end{align*}
Now we denote the shift operator in the path space by $\theta_t: D \rightarrow D$, i.e. it holds $(\xi \circ \theta_t)_s = \xi_{s+t}$. With the tower property of the conditional expectation and the Markov property in the version including the shift-operator (see e.g. \cite{Chu_Wal_01} p. 8) it holds:
\begin{align*}
&\quad \mP^x(\Lambda, t< T_{(-(1+\delta),1+\delta)}, t < \underline m ,\xi_{\underline m} \in (1,1+\eps)) \\
&= \mE^x \Big[ \1_\Lambda  \mE^x \big[\1_{\{ t< T_{(-(1+\delta),1+\delta)} \}} \1_{\{ \xi_{\underline m} \in (1,1+\eps) \}} \,|\, \cF_t \big] \Big] \\
&= \mE^x \Big[ \1_\Lambda  \mE^x \big[\1_{\{ t< T_{(-(1+\delta),1+\delta)} \}}(\1_{\{ \xi_{\underline m} \in (1,1+\eps) \}} \circ \theta_t )\,|\, \cF_t \big] \Big] \\
&= \mE^x \Big[ \1_\Lambda \1_{\{ t< T_{(-(1+\delta),1+\delta)} \}} \mE^x \big[\1_{\{ \xi_{\underline m} \in (1,1+\eps) \}} \circ \theta_t \,|\, \cF_t \big] \Big] \\
&= \mE^x \Big[ \1_\Lambda  \1_{\{ t< T_{(-(1+\delta),1+\delta)} \}} \mP^{\xi_t}(\xi_{\underline m} \in (1,1+\eps)) \Big].
\end{align*}
Hence, we have 
\begin{align*}
&\quad\mP^x(\Lambda, t< T_{(-(1+\delta),1+\delta)} ,\xi_{\underline m} \in (1,1+\eps))\\
&= \mE^x \Big[ \1_\Lambda  \1_{\{ t< T_{(-(1+\delta),1+\delta)} \}} \mP^{\xi_t}(\xi_{\underline m} \in (1,1+\eps)) \Big], \quad |x|>\delta>\eps.
\end{align*}
With the help of this application of the Markov property we obtain
\begin{align*}
&\quad \mP^x(\Lambda, t< T_{(-(1+\delta),1+\delta)} \, |\,\xi_{\underline m} \in (1,1+\eps))\\
&= \frac{\mP^x(\Lambda, t< T_{(-(1+\delta),1+\delta)},\xi_{\underline m} \in (1,1+\eps))}{\mP^x(\xi_{\underline m} \in (1,1+\eps))} \\
&= \mE^x \Big[ \1_\Lambda \1_{\{ t< T_{(-(1+\delta),1+\delta)}\}} \frac{\mP^{\xi_t}(\xi_{\underline m} \in (1,1+\eps))}{\mP^x(\xi_{\underline m} \in (1,1+\eps))} \Big].
\end{align*}
Now we would like to replace the ratio inside the expectation by $v_1(\xi_t){/}v_1(x)$ with Proposition \ref{prop_as_<1} when $\eps$ tends to $0$. For that we need to argue why we can move the $\eps$-limit inside the integral. Without loss of generality we assume $|x|>1+\delta>1+\eps$. Note that for $y>1+\delta$ we have again with Proposition 1.1 of \cite{Kyp_Riv_Sen_01}:
\begin{align*}
\mP^{y}(|\xi_{\underline{m}}| \in (1,1+\eps)) &= 2 \frac{2^{-\alpha}\Gamma(1-\alpha\rho)}{\Gamma(1-\alpha)\Gamma(\alpha\hat\rho)} \int\limits_{\frac{y}{1+\eps}}^{y} \psi_{\alpha\rho}(z) \, \dd z \\
&\leq  \frac{2^{1-\alpha}\Gamma(1-\alpha\rho)}{\Gamma(1-\alpha)\Gamma(\alpha\hat\rho)} \Big(y-\frac{y}{1+\eps}\Big) \psi_{\alpha\rho}\Big(\frac{y}{1+\eps}\big)\\
&=  \frac{2^{1-\alpha}\Gamma(1-\alpha\rho)}{\Gamma(1-\alpha)\Gamma(\alpha\hat\rho)} \frac{y\eps}{1+\eps} \psi_{\alpha\rho}\Big(\frac{y}{1+\eps}\Big)\\
&=  \frac{2^{1-\alpha}\Gamma(1-\alpha\rho)}{\Gamma(1-\alpha)\Gamma(\alpha\hat\rho)} \frac{\eps}{2\sin(\pi\alpha\hat\rho)} v\Big(\frac{y}{1+\eps}\Big).
\end{align*}
Now let $\eps$ be so small that $\frac{1+\delta}{1+\eps}>1+\frac{\delta}{2}$ and define
$$C_\delta := \sup_{|u| > 1+\frac{\delta}{2}} v(u),$$
which is finite because of the properties of $v$. So we can estimate on the event $\lc t< T_{[-1,1+\delta]}, \xi_t \geq 1+\delta \rc$:
\begin{align*}
\frac{1}{\eps} \mP^{\xi_t}(\xi_{\underline m} \in (1,1+\eps)) &\leq \frac{2^{-\alpha}\Gamma(1-\alpha\rho)}{\Gamma(1-\alpha)\Gamma(\alpha\hat\rho)\sin(\pi\alpha\hat\rho)} v\big(\frac{\xi_t}{1+\eps}\big)\\
&\leq \frac{2^{-\alpha}\Gamma(1-\alpha\rho)}{\Gamma(1-\alpha)\Gamma(\alpha\hat\rho)\sin(\pi\alpha\hat\rho)} C_\delta.
\end{align*}
On $\lc t< T_{(-(1+\delta),1+\delta)}, \xi_t \leq -(1+\delta) \rc$ an analogous argumentation shows
\begin{align*}
\frac{1}{\eps}\mP^{\xi_t}(\xi_{\underline m} \in (1,1+\eps)) &\leq \frac{2^{-\alpha}\Gamma(1-\alpha\hat\rho)}{\Gamma(1-\alpha)\Gamma(\alpha\rho)\sin(\pi\alpha\rho)} C_\delta.
\end{align*}
So we can use dominated convergence as follows:
\begin{align*}
&\quad\lim_{\eps \searrow 0} \mP^x(\Lambda, t< T_{(-(1+\delta),1+\delta)} \, |\,\xi_{\underline m} \in (1,1+\eps)) \\
&= \lim_{\eps \searrow 0}\frac{\eps}{{\mP^x(\xi_{\underline m} \in (1,1+\eps))}}  \lim_{\eps \searrow 0} \mE^x \Big[ \1_\Lambda \1_{\{ t< T_{(-(1+\delta),1+\delta)}\}}  \frac{\mP^{\xi_t}(\xi_{\underline m} \in (1,1+\eps))}{\eps} \Big] \\
&= \lim_{\eps \searrow 0}\frac{\eps}{{\mP^x(\xi_{\underline m} \in (1,1+\eps))}}  \mE^x \Big[ \1_\Lambda \1_{\{ t< T_{(-(1+\delta),1+\delta)}\}} \lim_{\eps \searrow 0} \frac{\mP^{\xi_t}(\xi_{\underline m} \in (1,1+\eps))}{\eps} \Big] \\ 
&= \mE^x \Big[ \1_\Lambda \1_{\{ t< T_{(-(1+\delta),1+\delta)}\}} \frac{v_1(\xi_t)}{v_1(x)} \Big]. \\
&= \mP_{v_1}^x(\Lambda, t< T_{(-(1+\delta),1+\delta)}).
\end{align*}
In the last step we used Proposition \ref{prop_as_<1}. This proves Theorem \ref{thm_cond_v1_<1}.\smallskip

The proof of Theorem \ref{thm_cond_v_<1} is similar. Applying the Markov property in the shift-operator-version we get 
\begin{align*}
&\quad\mP^x(\Lambda, t< T_{(-(1+\delta),1+\delta)} \, |\,|\xi_{\underline m}| \in (1,1+\eps))\\
&= \frac{\mP^x(\Lambda, t< T_{(-(1+\delta),1+\delta)},|\xi_{\underline m}| \in (1,1+\eps))}{\mP^x(|\xi_{\underline m}| \in (1,1+\eps))} \\
&= \mE^x \Big[ \1_\Lambda \1_{\{ t< T_{(-(1+\delta),1+\delta)}\}} \frac{\mP^{\xi_t}(|\xi_{\underline m}| \in (1,1+\eps))}{\mP^x(|\xi_{\underline m}| \in (1,1+\eps))} \Big].
\end{align*}
In the proof of Theorem \ref{thm_cond_v1_<1} we already found an integrable dominating function for $\mP^{\xi_t}(|\xi_{\underline m}| \in (1,1+\eps)) /\eps$. So we can use dominated convergence as follows:
\begin{align*}
&\quad \lim_{\eps \searrow 0} \mP^x(\Lambda, t< T_{(-(1+\delta),1+\delta)} \, |\,|\xi_{\underline m}| \in (1,1+\eps)) \\
&= \lim_{\eps \searrow 0}\frac{\eps}{{\mP^x(|\xi_{\underline m}| \in (1,1+\eps))}}  \lim_{\eps \searrow 0} \mE^x \Big[ \1_\Lambda \1_{\{ t< T_{(-(1+\delta),1+\delta)}\}}  \frac{\mP^{\xi_t}(|\xi_{\underline m}| \in (1,1+\eps))}{\eps} \Big] \\
&= \lim_{\eps \searrow 0}\frac{\eps}{{\mP^x(|\xi_{\underline m}| \in (1,1+\eps))}}  \mE^x \Big[ \1_\Lambda \1_{\{ t< T_{(-(1+\delta),1+\delta)}\}} \lim_{\eps \searrow 0} \frac{\mP^{\xi_t}(|\xi_{\underline m}| \in (1,1+\eps))}{\eps} \Big] \\ 
&= \mE^x \Big[ \1_\Lambda \1_{\{ t< T_{(-(1+\delta),1+\delta)}\}} \frac{v(\xi_t)}{v(x)} \Big]. \\
&= \mP_{v}^x(\Lambda, t< T_{(-(1+\delta),1+\delta)}),
\end{align*}
where we used Proposition \ref{prop_as_<1} in the last equality.
\end{proof}

\subsubsection{The case $\alpha \geq 1$}
The strategy for $\alpha \geq 1$ is as in the case $\alpha <1$. First we need a relation between $v_1$ and the asymptotic probability we want to condition on. This event looks a bit different from the one in the case $\alpha<1$.
\begin{proposition} \label{prop_as_geq1}
Let $\xi$ be an $\alpha$-stable process with $\alpha \in [1,2)$ and both sided jumps, then
\begin{align*}
\frac{1-\alpha\hat\rho}{2^{\alpha\rho} \pi} v_1(x) &= \lim_{\eps \searrow 0} \frac{1}{\eps^{1-\alpha\hat\rho}}\mP^x(\xi_{T_{(-(1+\eps),1+\eps)}} \in (1,1+\eps)), \quad x \in \mR \setminus [-1,1].
\end{align*}
\end{proposition}
\begin{proof}
Using the scaling property and Theorem 1.1 of \cite{Kyp_Par_Wat_01} we get for $x>1+\eps$:
\begin{align}\label{help_111}
\begin{split}
&\quad\frac{\pi}{\sin(\pi\alpha\hat\rho)} \mP^x( \xi_{T_{(-(1+\eps),1+\eps)}} \in (1,1+\eps))\\
&= \frac{\pi}{\sin(\pi\alpha\hat\rho)} \mP^{\frac{x}{1+\eps}}\Big( \xi_{T_{(-1,1)}} \in \big(\frac{1}{1+\eps},1\big)\Big)\\
&= \big(\frac{x}{1+\eps}+1\big)^{\alpha\rho} \big(\frac{x}{1+\eps}-1\big)^{\alpha\hat\rho} \int\limits_{\frac{1}{1+\eps}}^1 (1+y)^{-\alpha\rho} (1-y)^{-\alpha\hat\rho} \big( \frac{x}{1+\eps}-y \big)^{-1} \, \dd y \\
& \quad  -(\alpha-1) \int\limits_1^{\frac{x}{1+\eps}} \psi_{\alpha\rho}(u)\, \dd u \int\limits_{\frac{1}{1+\eps}}^1 (1+y)^{-\alpha\rho} (1-y)^{-\alpha\hat\rho} \, \dd y.
\end{split}
\end{align}
With l'Hopital's rule and the integration rule of Leibnitz we see:
\begin{align}\label{help_112}
\begin{split}
&\quad\lim_{\eps \searrow 0} \frac{1}{\eps^{1-\alpha\hat\rho}} \int\limits_{\frac{1}{1+\eps}}^1 (1+y)^{-\alpha\rho} (1-y)^{-\alpha\hat\rho} \big( \frac{x}{1+\eps}-y \big)^{-1} \, \dd y \\
&=\lim_{\eps \searrow 0} \frac{\eps^{\alpha\hat\rho}}{1-\alpha\hat\rho} \frac{1}{(1+\eps)^2} \big(1+\frac{1}{1+\eps}\big)^{-\alpha\rho} \big(1-\frac{1}{1+\eps}\big)^{-\alpha\hat\rho} \big(\frac{x-1}{1+\eps}\big)^{-1} \\
&\quad \quad  + \lim_{\eps \searrow 0}  \frac{\eps^{\alpha\hat\rho}}{1-\alpha\hat\rho} \int\limits_{\frac{1}{1+\eps}}^1 (1+y)^{-\alpha\rho} (1-y)^{-\alpha\hat\rho}\frac{x}{(x-y(1+\eps))^2} \, \dd y \\
&= \frac{2^{-\alpha\rho}}{1-\alpha\hat\rho} (x-1)^{-1}
\end{split}
\end{align}
and further, 
\begin{align}\label{help_113}
\begin{split}
&\quad\lim_{\eps \searrow 0} \frac{1}{\eps^{1-\alpha\hat\rho}} \int\limits_{\frac{1}{1+\eps}}^1 (1+y)^{-\alpha\rho} (1-y)^{-\alpha\hat\rho} \, \dd y\\
&= \lim_{\eps \searrow 0} \frac{\eps^{\alpha\hat\rho}}{1-\alpha\hat\rho} \frac{1}{(1+\eps)^2} \big(1+\frac{1}{1+\eps}\big)^{-\alpha\rho} \big(1-\frac{1}{1+\eps}\big)^{-\alpha\hat\rho}\\
&= \frac{2^{-\alpha\rho}}{1-\alpha\hat\rho}.
\end{split}
\end{align}
Now we plug in \eqref{help_112} and \eqref{help_113} in \eqref{help_111} and get
\begin{align*}
&\quad\lim_{\eps \searrow 0} \frac{\pi}{\eps^{1-\alpha\hat\rho}} \mP^x( \xi_{T_{(-(1+\eps),1+\eps)}} \in (1,1+\eps))\\
&=\frac{2^{-\alpha\rho}}{1-\alpha\hat\rho} \sin(\pi\alpha\hat\rho)\Big[(x+1)^{\alpha\rho} (x-1)^{\alpha\hat\rho} (x-1)^{-1} - (\alpha-1) \int\limits_1^{x} \psi_{\alpha\rho}(u)\, \dd u \Big] \\
&= \frac{2^{-\alpha\rho}}{1-\alpha\hat\rho} v_1(x).
\end{align*}
For $x<-1$ we note that
$$\mP^x( \xi_{T_{(-(1+\eps),1+\eps)}} \in (1,1+\eps)) = \hat{\mP}^{|x|}( \xi_{T_{(-(1+\eps),1+\eps)}} \in (-(1+\eps),-1)),$$
use again Theorem 1.1 of \cite{Kyp_Par_Wat_01} and do a similar calculation as above to deduce 
\begin{align*}
&\quad \frac{\pi}{\sin(\pi\alpha\rho)} \mP^x( \xi_{T_{(-(1+\eps),1+\eps)}} \in (1,1+\eps))\\
&= \big(\frac{|x|}{1+\eps}+1\big)^{\alpha\hat\rho} \big(\frac{|x|}{1+\eps}-1\big)^{\alpha\rho} \int\limits_{-1}^{-\frac{1}{1+\eps}} (1+y)^{-\alpha\hat\rho} (1-y)^{-\alpha\rho} \big( \frac{|x|}{1+\eps}-y \big)^{-1} \, \dd y \\
& \quad  -(\alpha-1) \int\limits_1^{\frac{|x|}{1+\eps}} \psi_{\alpha\hat\rho}(u)\, \dd u \int\limits_{-1}^{-\frac{1}{1+\eps}} (1+y)^{-\alpha\hat\rho} (1-y)^{-\alpha\rho} \, \dd y.
\end{align*}
A substitution on the integrals and the same limiting arguments as in the case $x>1$ show
\begin{align*}
&\quad \lim_{\eps \searrow 0} \frac{\pi}{\eps^{1-\alpha\hat\rho}} \mP^x( \xi_{T_{(-(1+\eps),1+\eps)}} \in (1,1+\eps))\\
&=\frac{2^{-\alpha\rho}}{1-\alpha\hat\rho} \sin(\pi\alpha\rho)\Big[(|x|+1)^{\alpha\hat\rho} (|x|-1)^{\alpha\rho} (|x|+1)^{-1} - (\alpha-1) \int\limits_1^{|x|} \psi_{\alpha\hat\rho}(u)\, \dd u \Big] \\
&= \frac{2^{-\alpha\rho}}{1-\alpha\hat\rho} v_1(x).
\end{align*}
\end{proof}

\begin{proof}[Proof of Theorems \ref{thm_cond_v1_geq1} and \ref{thm_cond_v_geq1}]
First we note by a similar application of the Markov property as in the proof of Theorem \ref{thm_cond_v1_<1}:
\begin{align*}
&\quad \mP^x(\Lambda, t< T_{(-(1+\delta),1+\delta)} ,\xi_{T_{(-(1+\eps),1+\eps)}} \in (1,1+\eps)) \\
&= \mE^x \Big[ \1_\Lambda \1_{\{ t< T_{(-(1+\delta),1+\delta)} \}} \mP^{\xi_t}(\xi_{T_{(-(1+\eps),1+\eps)}} \in (1,1+\eps)) \Big]
\end{align*}
and hence,
\begin{align*}
&\quad \mP^x(\Lambda, t< T_{(-(1+\delta),1+\delta)} \,|\, \xi_{T_{(-(1+\eps),1+\eps)}} \in (1,1+\eps)) \\
&= \mE^x \Big[ \1_\Lambda \1_{\{ t< T_{(-(1+\delta),1+\delta)} \}} \frac{\mP^{\xi_t}(\xi_{T_{(-(1+\eps),1+\eps)}} \in (1,1+\eps))}{\mP^{x}(\xi_{T_{(-(1+\eps),1+\eps)}} \in (1,1+\eps))} \Big].
\end{align*}
Again we want to move the $\eps$-limit inside the integral and use Proposition \ref{prop_as_geq1}. First we use \eqref{help_111}:
\begin{align*}
&\quad\frac{\pi}{\sin(\pi\alpha\hat\rho)} \mP^y( \xi_{T_{(-(1+\eps),1+\eps)}} \in (1,1+\eps))\\
&= \big(\frac{y}{1+\eps}+1\big)^{\alpha\rho} \big(\frac{y}{1+\eps}-1\big)^{\alpha\hat\rho} \int\limits_{\frac{1}{1+\eps}}^1 (1+u)^{-\alpha\rho} (1-u)^{-\alpha\hat\rho} \big( \frac{y}{1+\eps}-u \big)^{-1} \, \dd u \\
& \quad  -(\alpha-1) \int\limits_1^{\frac{y}{1+\eps}} \psi_{\alpha\rho}(w)\, \dd w \int\limits_{\frac{1}{1+\eps}}^1 (1+u)^{-\alpha\rho} (1-u)^{-\alpha\hat\rho} \, \dd u \\
&\leq  \Big[ \big(\frac{y}{1+\eps}+1\big)^{\alpha\rho} \big(\frac{y}{1+\eps}-1\big)^{\alpha\hat\rho-1} - (\alpha-1) \int\limits_1^{\frac{y}{1+\eps}} \psi_{\alpha\rho}(w)\, \dd w \Big] \int\limits_{\frac{1}{1+\eps}}^1 (1+u)^{-\alpha\rho} (1-u)^{-\alpha\hat\rho} \, \dd u\\
&= \frac{1}{\sin(\pi\alpha\hat\rho)}v_1\big(\frac{y}{1+\eps}\big) \int\limits_{\frac{1}{1+\eps}}^1 (1+u)^{-\alpha\rho} (1-u)^{-\alpha\hat\rho} \, \dd u.
\end{align*}
Further,
\begin{align*}
\eps^{\alpha\hat\rho-1} \int\limits_{\frac{1}{1+\eps}}^1 (1+u)^{-\alpha\rho} (1-u)^{-\alpha\hat\rho} \, \dd u &\leq \eps^{\alpha\hat\rho-1}  \int\limits_{\frac{1}{1+\eps}}^1(1-u)^{-\alpha\hat\rho} \, \dd u \\
&= \frac{\eps^{\alpha\hat\rho-1}}{1-\alpha\hat\rho} \Big(\frac{\eps}{1+\eps}\Big)^{1-\alpha\hat\rho} \\
&\leq \frac{1}{1-\alpha\hat\rho}.
\end{align*}
Let be $\eps$ so small that $\frac{1+\delta}{1+\eps} > 1+\frac{\delta}{2}$ and define
$$C_\delta = \sup_{|u| \geq 1+\frac{\delta}{2}} v_1(u).$$
Then it follows
\begin{align*}
\frac{\pi}{\eps^{1-\alpha\hat\rho}} \mP^y( \xi_{T_{(-(1+\eps),1+\eps)}} \in (1,1+\eps)) &\leq \frac{1}{1-\alpha\hat\rho} v_1\Big(\frac{y}{1+\eps}\Big)
\leq \frac{C_\delta}{1-\alpha\hat\rho}.
\end{align*}
Similarly, we get for $y<-(1+\eps)$:
\begin{align*}
\frac{\pi}{\eps^{1-\alpha\hat\rho}} \mP^y( \xi_{T_{(-(1+\eps),1+\eps)}} \in (1,1+\eps)) &\leq \frac{1}{1-\alpha\hat\rho} v_1\Big(\frac{y}{1+\eps}\Big)\\
&\leq \frac{C_\delta}{1-\alpha\hat\rho}.
\end{align*}
So we can apply dominated convergence to deduce
\begin{align*}
&\quad\lim_{\eps \searrow 0} \mP^x(\Lambda, t< T_{(-(1+\delta),1+\delta)} \, |\xi_{T_{(-(1+\eps),1+\eps)}} \in (1,1+\eps)) \\
&= \lim_{\eps \searrow 0}\frac{\eps^{1-\alpha\hat\rho}}{{\mP^x(\xi_{T_{(-(1+\eps),1+\eps)}} \in (1,1+\eps))}} \\
&\quad   \times \lim_{\eps \searrow 0} \mE^x \Big[ \1_\Lambda \1_{\{ t< T_{(-(1+\delta),1+\delta)}\}}  \frac{\mP^{\xi_t}(\xi_{T_{(-(1+\eps),1+\eps)}} \in (1,1+\eps))}{\eps^{1-\alpha\hat\rho}} \Big] \\
&= \lim_{\eps \searrow 0}\frac{\eps^{1-\alpha\hat\rho}}{{\mP^x(\xi_{T_{(-(1+\eps),1+\eps)}} \in (1,1+\eps))}} \\
&\quad   \times \mE^x \Big[ \1_\Lambda \1_{\{ t< T_{(-(1+\delta),1+\delta)}\}} \lim_{\eps \searrow 0} \frac{\mP^{\xi_t}(\xi_{T_{(-(1+\eps),1+\eps)}} \in (1,1+\eps))}{\eps^{1-\alpha\hat\rho}} \Big] \\ 
&= \mE^x \Big[ \1_\Lambda \1_{\{ t< T_{(-(1+\delta),1+\delta)}\}} \frac{v_1(\xi_t)}{v_1(x)} \Big] \\
&= \mP_{v_1}^x(\Lambda, t< T_{(-(1+\delta),1+\delta)}),
\end{align*}
where we used Proposition \ref{prop_as_geq1} in the second last equality. This finishes the proof of Theorem \ref{thm_cond_v1_geq1}.\smallskip

 To prove Theorem \ref{thm_cond_v_geq1} we first not that one can show analogously to the proof of Proposition \ref{prop_as_geq1}:
$$\lim_{\eps \searrow 0} \eps^{\alpha\rho-1}\mP^{x}(\xi_{T_{(-(1+\eps),1+\eps)}} \in (-(1+\eps),-1)) = \frac{1-\alpha\rho}{2^{\alpha\hat\rho} \pi}v_{-1}(x), \quad x \notin [-1,1].$$
We assume without loss of generality $\rho \leq \hat\rho$ (i.e. $\rho \leq 1/2$) and in particular it holds that
\begin{align*}
&\quad \lim_{\eps \searrow 0} \eps^{\alpha\hat\rho-1}\mP^{x}(\xi_{T_{(-(1+\eps),1+\eps)}} \in (-(1+\eps),-1))\\
&= \lim_{\eps \searrow 0} \eps^{\alpha(\hat\rho-\rho)} \eps^{\alpha\rho-1}\mP^{x}(\xi_{T_{(-(1+\eps),1+\eps)}} \in (-(1+\eps),-1))\\
&= 0.
\end{align*}
It follows that
\begin{align*}
&\quad\lim_{\eps \searrow 0} \frac{\mP^{\xi_t}(|\xi_{T_{(-(1+\eps),1+\eps)}}| \in (1,1+\eps))} {\mP^{x}(|\xi_{T_{(-(1+\eps),1+\eps)}}| \in (1,1+\eps))} \\
&=\lim\limits_{\eps \searrow 0} \frac{ \mP^{\xi_t}(\xi_{T_{(-(1+\eps),1+\eps)}} \in (1,1+\eps))+ \mP^{\xi_t}(\xi_{T_{(-(1+\eps),1+\eps)}} \in (-(1+\eps),-1))}{\mP^{x}(\xi_{T_{(-(1+\eps),1+\eps)}} \in (1,1+\eps))+\mP^{x}(\xi_{T_{(-(1+\eps),1+\eps)}} \in (-(1+\eps),-1))} \\
&=\lim\limits_{\eps \searrow 0} \frac{\eps^{\alpha\hat\rho-1}\mP^{\xi_t}(\xi_{T_{(-(1+\eps),1+\eps)}} \in (1,1+\eps))+ \eps^{\alpha\hat\rho-1} \mP^{\xi_t}(\xi_{T_{(-(1+\eps),1+\eps)}} \in (-(1+\eps),-1))}{\eps^{\alpha\hat\rho-1}\mP^{x}(\xi_{T_{(-(1+\eps),1+\eps)}} \in (1,1+\eps))+\eps^{\alpha\hat\rho-1}\mP^{x}(\xi_{T_{(-(1+\eps),1+\eps)}} \in (-(1+\eps),-1))}\\
&=\lim\limits_{\eps \searrow 0} \frac{\eps^{\alpha\hat\rho-1}\mP^{\xi_t}(\xi_{T_{(-(1+\eps),1+\eps)}} \in (1,1+\eps))}{\eps^{\alpha\hat\rho-1}\mP^{x}(\xi_{T_{(-(1+\eps),1+\eps)}} \in (1,1+\eps))}\\
&= \frac{v_1(\xi_t)}{v_1(x)}.
\end{align*}
For $\rho>1/2$, following the same argument, the first summands vanish instead of the second. To finish the proof of Theorem \ref{thm_cond_v_geq1} the dominated convergence argument can be transferred from the proof of Theorem \ref{thm_cond_v1_geq1}.
\end{proof}

\subsubsection{The alternative characterisation for $\alpha>1$}
As before we start with the needed asymptotic probability of the event we want to condition on which is, as already mentioned, the same as in the case $\alpha<1$ but under the law of the process conditioned to avoid $0$.
\begin{proposition}\label{prop_as_altern_>1}
Let $\xi$ be an $\alpha$-stable process with $\alpha \in (1,2)$ and both sided jumps, then
\begin{align}\label{eq_as1_>1}
\frac{\alpha-1}{2} v_1(x) &= \lim\limits_{\eps \searrow 0} \frac{e(x)}{\eps}\mP_\circ^{x}(\xi_{\underline{m}} \in (1,1+\eps)), \quad x \notin [-1,1],
\end{align}
and
\begin{align}\label{eq_as_>1}
\frac{\alpha-1}{2} v(x) &= \lim\limits_{\eps \searrow 0} \frac{e(x)}{\eps}\mP_\circ^{x}(|\xi_{\underline{m}}| \in (1,1+\eps)), \quad x \notin [-1,1].
\end{align}
\end{proposition}
\begin{proof}
We use the so-called point of furthest reach before hitting $0$. Let $\overline{m}$ be the time such that $|\xi_t| \leq |\xi_{\overline{m}}|$ for all $t \leq T_0$. The Riesz-Bogdan-Żak (see Bogdan and Żak \cite{Bog_Zak_01} for symmetric stable processes, Kyprianou \cite{Kyp_02} for general stable processes and Alili et al. \cite{Ali_Cha_Gra_Zak_01} for self-similar Markov processes) tells us that the process conditioned to avoid $0$ is the spatial inverse of the original (i.e. not $h$-transformed) dual process including a certain time-change. Since the time change does not play any role for the value $\xi_{\overline{m}}$ we can extract the distribution of the point of closest reach of the process conditioned to avoid $0$ from the distribution of the point of furthest reach of the original dual process, i.e.
$$\mP_\circ^{x}(\xi_{\underline{m}} \in (1,1+\eps)) = \hat{\mP}^{\frac{1}{x}}\Big(\xi_{\overline{m}} \in \Big(\frac{1}{1+\eps},1 \Big)\Big).$$
Combining this with Proposition 1.2 of Kyprianou et al. \cite{Kyp_Riv_Sen_01} where one can find an explicit expression for the distribution of the point of furthest reach before hitting $0$, we get for $x>1$:
\begin{align*}
&\,\frac{2}{\alpha-1}\mP_\circ^{x}(\xi_{\underline{m}} \in (1,1+\eps))\\
=&\,\int\limits_{\frac{1}{x} \lor \frac{1}{1+\eps}}^1 u^{-\alpha} \Big[\big(u+\frac{1}{x}\big)^{\alpha\rho}\big(u-\frac{1}{x}\big)^{\alpha\hat{\rho}-1} -(\alpha-1) x^{1-\alpha} \int\limits_1^{ux} \psi_{\alpha\rho}(w) \, \dd w \Big] \, \dd u \\
=&\, \int\limits_{1 \lor \frac{x}{1+\eps}}^x \frac{1}{x} \left(\frac{x}{u}\right)^{\alpha} \Big[\big(\frac{u}{x}+\frac{1}{x}\big)^{\alpha\rho}\big(\frac{u}{x}-\frac{1}{x}\big)^{\alpha\hat{\rho}-1} -(\alpha-1)x^{1-\alpha} \int\limits_1^{u} \psi_{\alpha\rho}(w) \, \dd w \Big] \, \dd u \\
=&\, \int\limits_{1 \lor \frac{x}{1+\eps}}^x u^{-\alpha} \Big[(u+1)^{\alpha\rho}(u-1)^{\alpha\hat{\rho}-1} -(\alpha-1) \int\limits_1^{u} \psi_{\alpha\rho}(w) \, \dd w \Big] \, \dd u \\
=&\, \int\limits_{1 \lor \frac{x}{1+\eps}}^x u^{-\alpha} \Big[(u+1)\psi_{\alpha\rho}(u) -(\alpha-1) \int\limits_1^{u} \psi_{\alpha\rho}(w) \, \dd w \Big] \, \dd u \\
=& \frac{1}{\sin(\pi\alpha\hat\rho)} \int\limits_{1 \lor \frac{x}{1+\eps}}^x u^{-\alpha} v_1(u) \, \dd u.
\end{align*}
With l'Hopital's rule we get
\begin{align*}
\lim_{\eps \searrow 0} \frac{2}{\alpha-1} \frac{1}{\eps} \mP_\circ^{x}(\xi_{\underline{m}} \in (1,1+\eps)) &= \frac{1}{\sin(\pi\alpha\hat\rho)} \lim_{\eps \searrow 0} \Big[\frac{x}{(1+\eps)^2} \Big(\frac{x}{1+\eps}\Big)^{-\alpha} v_1\Big(\frac{x}{1+\eps}\Big) \Big] \\
&= \frac{1}{\sin(\pi\alpha\hat\rho)} x^{1-\alpha} v_1(x)\\
&= \frac{v_1(x)}{e(x)}.
\end{align*}
This shows \eqref{eq_as1_>1} for $x>1$. For $x<-1$ the equality \eqref{eq_as1_>1} follows similarly. \smallskip

To show the second claim we use
\begin{align*}
\frac{\alpha-1}{2} v_{-1}(x) &= \lim\limits_{\eps \searrow 0} \frac{e(x)}{\eps}\mP_\circ^{x}(\xi_{\underline{m}} \in (-(1+\eps),-1)), \quad x \notin [-1,1],
\end{align*}
which follows from a computation similar to \eqref{eq_as1_>1}. Using that $v=v_1+v_{-1}$ the second claim follows.
\end{proof}

\begin{proof}[Proof of Theorems \ref{thm_cond_v1_altern_>1} and \ref{thm_cond_v_altern_>1}]
Since the process conditioned to avoid $0$ is a strong Markov process (this follows by general theory on $h$-transforms, see e.g. Chung and Walsh \cite{Chu_Wal_01}) we can use  arguments analogous to the case $\alpha<1$ to obtain, for all $x \notin [-1,1]$,
\begin{align*}
&\quad\mP_\circ^x(\Lambda, t< T_{(-(1+\delta),1+\delta)} \, |\,\xi_{\underline m} \in (1,1+\eps))\\
&= \mE_\circ^x \la \1_\Lambda \1_{\lc t< T_{(-(1+\delta),1+\delta)}\rc} \frac{\mP_\circ^{\xi_t}(\xi_{\underline m} \in (1,1+\eps))}{\mP_\circ^x(\xi_{\underline m} \in (1,1+\eps))} \ra.
\end{align*}
In the proof of Proposition \ref{prop_as_altern_>1} we have already seen that
$$\mP_\circ^{y}(\xi_{\underline{m}} \in (1,1+\eps)) = \frac{\alpha-1}{2 \sin(\pi\alpha\hat\rho)} \int\limits_{1 \lor \frac{y}{1+\eps}}^y u^{-\alpha}v_{1}(u) \, \dd u$$
for $y>1+\eps$. Analogously we can show
$$\mP_\circ^{y}(\xi_{\underline{m}} \in (-(1+\eps),-1)) = \frac{\alpha-1}{2 \sin(\pi\alpha\hat\rho)} \int\limits_{1 \lor \frac{y}{1+\eps}}^y u^{-\alpha}v_{-1}(u) \, \dd u$$ 
for $y>1+\eps$ and hence, we have
$$\mP_\circ^{y}(|\xi_{\underline{m}}| \in (1,1+\eps)) = \frac{\alpha-1}{2\sin(\pi\alpha\hat\rho)} \int\limits_{1 \lor \frac{y}{1+\eps}}^y u^{-\alpha}v(u) \, \dd u$$
for $y>1+\eps$. Now we fix $\delta>0$ and assume that $\eps$ is so small that $\frac{1+\delta}{1+\eps} \geq 1+\frac{\delta}{2}$. We define again $C_\delta:= \sup_{|u| \geq 1+\frac{\delta}{2}} v(u)$ which is finite. Note that, for $y>1+\delta$, we have:
\begin{align*}
\mP_\circ^{y}(|\xi_{\underline{m}}| \in (1,1+\eps)) &=\frac{\alpha-1}{2\sin(\pi\alpha\hat\rho)}\int\limits_{\frac{y}{1+\eps}}^y u^{-\alpha} v(u) \, \dd u \\
&\leq \frac{\alpha-1}{2\sin(\pi\alpha\hat\rho)} \frac{y\eps}{1+\eps} \Big(\frac{y}{1+\eps}\Big)^{-\alpha} \sup_{u \in [\frac{y}{1+\eps},\infty)} v(u)\\
&\leq \frac{C_\delta(\alpha-1)}{2\sin(\pi\alpha\hat\rho)} \frac{\eps}{(1+\eps)^{1-\alpha}}y^{1-\alpha}\\
&\leq \frac{C_\delta(\alpha-1)}{2\sin(\pi\alpha\hat\rho)} \eps (1+\delta)^{\alpha-1 }y^{1-\alpha}.
\end{align*}
So we can estimate on $\{ t< T_{[-(1+\delta),1+\delta]}, \xi_t > 1 \}$:
\begin{align*}
\frac{\mP^{\xi_t}(\xi_{\underline m} \in (1,1+\eps))}{\eps} &\leq \frac{C_\delta(\alpha-1)}{2\sin(\pi\alpha\hat\rho)} (1+\delta)^{\alpha-1 }\xi_t^{1-\alpha}.
\end{align*}
On $\{ t< T_{[-(1+\delta),1+\delta]}, \xi_t <- 1 \}$ an analogous argument shows
\begin{align*}
\frac{\mP^{\xi_t}(\xi_{\underline m} \in (1,1+\eps))}{\eps} &\leq \frac{C_\delta(\alpha-1)}{2\sin(\pi\alpha\rho)} (1+\delta)^{\alpha-1 }|\xi_t|^{1-\alpha}.
\end{align*}
Further it holds that
\begin{align*}
\frac{1}{\sin(\pi\alpha\hat\rho)}\mE_\circ^x \big[ \1_\Lambda \1_{\{ t< T_{[-(1+\delta),1+\delta]}, \xi_t > 1 \}} \xi_t^{1-\alpha} \big] 
&= \frac{1}{e(x)} \mE^x \big[ \1_\Lambda \1_{\{ t< T_{[-(1+\delta),1+\delta]}, \xi_t > 1 \}} \xi_t^{\alpha-1} \xi_t^{1-\alpha} \big] \\
&\leq \frac{1}{e(x)}
\end{align*}
and, analogously,
$$\frac{1}{\sin(\pi\alpha\rho)}\mE_\circ^x \big[ \1_\Lambda \1_{\{ t< T_{[-(1+\delta),1+\delta]}, \xi_t < -1 \}} |\xi_t|^{1-\alpha} \big] \leq \frac{1}{e(x)}.$$
So we can use dominated convergence and the Markov property as follows:
\begin{align*}
&\quad\lim_{\eps \searrow 0} \mP_\circ^x(\Lambda, t< T_{[-(1+\delta),1+\delta]} \, |\,\xi_{\underline m} \in (1,1+\eps)) \\
&= \lim_{\eps \searrow 0}\frac{\eps}{{\mP_\circ^x(\xi_{\underline m} \in (1,1+\eps))}} \\
&\quad  \quad \times \lim_{\eps \searrow 0} \mE_\circ^x \Big[ \1_\Lambda \1_{\{ t< T_{[-(1+\delta),1+\delta]}\}}  \frac{\mP_\circ^{\xi_t}(\xi_{\underline m} \in (1,1+\eps))}{\eps} \Big] \\
&= \lim_{\eps \searrow 0}\frac{\eps}{{\mP_\circ^x(\xi_{\underline m} \in (1,1+\eps))}} \\
&\quad  \quad \times \mE_\circ^x \Big[ \1_\Lambda \1_{\{ t< T_{[-(1+\delta),1+\delta]}\}} \lim_{\eps \searrow 0} \frac{\mP_\circ^{\xi_t}(\xi_{\underline m} \in (1,1+\eps))}{\eps} \Big] \\
&= \mE_\circ^x \Big[ \1_\Lambda \1_{\{ t< T_{[-(1+\delta),1+\delta]}\}} \frac{e(x)v_1(\xi_t)}{e(\xi_t)v_1(x)} \Big] \\
&= \mE^x \Big[ \1_\Lambda \1_{\{ t< T_{[-(1+\delta),1+\delta]}\}} \frac{v_1(\xi_t)}{v_1(x)} \Big].
\end{align*}
In the second last second step we used Proposition \ref{prop_as_altern_>1}. Theorem \ref{thm_cond_v_altern_>1} can be proven similarly using the same dominating function.
\end{proof}

\textbf{Acknowledgements.} The authors would like to thank Dr. Watson for useful discussions on this topic.

\bibliographystyle{abbrvnat}
\bibliography{references}

\end{document}

%% file: basis_paper.tex
\newcommand{\mE}{\mathbb{E}}

\newcommand{\mP}{\mathbb{P}}
\newcommand{\mR}{\mathbb{R}}

\newcommand{\cB}{\mathcal{B}}

\newcommand{\cF}{\mathcal{F}}

\newcommand{\lc}{\left\lbrace}
\newcommand{\rc}{\right\rbrace}
\newcommand{\la}{\left\lbrack}
\newcommand{\ra}{\right\rbrack}
\newcommand{\eps}{\varepsilon}
\newcommand{\1}{\mathds{1}}
\newcommand\dd{\mathrm{d}}

\newtheorem {definition}{Definition}[section]
\newtheorem {theorem}[definition]{Theorem}
\newtheorem{proposition}[definition]{Proposition}
\newtheorem {lemma}[definition]{Lemma}
\newtheorem {corollary}[definition]{Corollary}

\newtheorem {remark}[definition]{Remark}

